\documentclass[a4paper,11pt]{amsart}

\usepackage[utf8]{inputenc}		
\usepackage[T1]{fontenc}
\usepackage[english]{babel}

\usepackage{tikz,tikz-3dplot, forloop}

\usepackage{amsfonts}			
\usepackage{amsmath}
\usepackage{amssymb}
\usepackage{amsthm}

\usepackage{graphicx}			
\usepackage{tikz}
	\usetikzlibrary{cd}

\usepackage{hyperref}			
	\hypersetup{colorlinks=false, urlcolor=black, linkcolor=black}

\usepackage{enumitem}			

\usepackage{color}				
\usepackage{float}

\newcommand{\W}{\mathscr{W}}
\newcommand{\dtext}{\textnormal d}
\usepackage{mathrsfs}
\newcommand{\onto}{\xrightarrow[]{{}_{\!\!\textnormal{onto\,\,}\!\!}}}

\newcommand{\N}{\mathbb{N}}						
\newcommand{\Z}{\mathbb{Z}}						
\newcommand{\Q}{\mathbb{Q}}						
\newcommand{\R}{\mathbb{R}}						
\newcommand{\C}{\mathbb{C}}						

\renewcommand{\S}{\mathbb{S}}					
\newcommand{\B}{\mathbb{B}}

\newcommand{\eps}{\varepsilon}					
\newcommand{\derham}{\text{dR}}				

\newcommand{\dd}								
	{\mathop{}\!\mathrm{d}}						
\newcommand{\ddn}[1]							
	{\mathop{}\!\mathrm{d^{#1}}}

\newcommand{\abs}[1]							
	{\left| #1 \right|}
\newcommand{\smallabs}[1]						
	{\lvert #1 \rvert}	
\newcommand{\norm}[1]							
	{\left\lVert #1 \right\rVert}	
\newcommand{\smallnorm}[1]						
	{\lVert #1 \rVert}						
\newcommand{\ip}[2]								
	{\left< #1 , #2 \right>}

\DeclareMathOperator{\id}{id}					
\DeclareMathOperator{\spt}{spt}					
				
\DeclareMathOperator{\sgn}{sgn}

\DeclareMathOperator{\im}{im}					

\newcommand{\push}[1]{{#1}_*\,}					

\newcommand{\hodge}{\mathtt{\star}\hspace{1pt}}

\newcommand{\cesobloc}{W^{d}_{\text{CE}, \text{loc}}}
\newcommand{\cesobc}{W^{d}_{\text{CE}, c}}

\newcommand{\cehom}[1]{H_{\text{CE}}^{#1}}
\newcommand{\cehomc}[1]{H_{\text{CE}, c}^{#1}}

\newcommand{\loc}{\mathrm{loc}}
\DeclareMathOperator{\osc}{osc}

\newcommand{\cH}{\mathcal{H}}

\newcommand{\cS}{\mathcal{S}}
\newcommand{\cT}{\mathcal{T}}
\newcommand{\cR}{\mathcal{R}}

\newtheorem{thm}{Theorem}[section]{\bf}{\it}
\newtheorem{lemma}[thm]{Lemma}
\newtheorem{prop}[thm]{Proposition}
\newtheorem{cor}[thm]{Corollary}
\newtheorem{qu}[thm]{Question}

\newtheorem{innercustomthm}{Theorem}[section]{\bf}{\it}
\newenvironment{customthm}[1]
{\renewcommand\theinnercustomthm{#1}\innercustomthm}
{\endinnercustomthm}

\newtheorem{innercustomprop}{Proposition}[section]{\bf}{\it}
\newenvironment{customprop}[1]
{\renewcommand\theinnercustomprop{#1}\innercustomprop}
{\endinnercustomprop}

\newtheorem{innercustomcor}{Corollary}[section]{\bf}{\it}
\newenvironment{customcor}[1]
{\renewcommand\theinnercustomcor{#1}\innercustomcor}
{\endinnercustomcor}

\theoremstyle{definition}

\theoremstyle{remark}

\numberwithin{equation}{section}

\begin{document}

\title{Fibers of monotone maps of finite distortion}

\author[I. Kangasniemi]{Ilmari Kangasniemi}
\address{Department of Mathematics, Syracuse University, Syracuse,
NY 13244, USA }
\email{kikangas@syr.edu}

\author[J. Onninen]{Jani Onninen}
\address{Department of Mathematics, Syracuse University, Syracuse,
NY 13244, USA and  Department of Mathematics and Statistics, P.O.Box 35 (MaD) FI-40014 University of Jyv\"askyl\"a, Finland
}
\email{jkonnine@syr.edu}


\date{\today}


\maketitle

\begin{abstract}
	We study topologically monotone surjective $W^{1,n}$-maps of finite distortion $f \colon \Omega \to \Omega'$, where $\Omega, \Omega' $ are domains in $\mathbb{R}^n$, $n \geq 2$. If the outer distortion function $K_f \in L_{\mathrm{loc}}^{p}(\Omega)$ with $p \geq n-1$, then any such map $f$ is known to be homeomorphic, and hence the fibers $f^{-1}\{y\}$ are singletons. We show that as the exponent of integrability $p$ of the distortion function $K_f$ increases in the range $1/(n-1) \leq p < n-1$, then the fibers $f^{-1}\{y\}$ of $f$ start satisfying increasingly strong homological limitations. We also give a Sobolev realization of a topological example by Bing of a monotone $f \colon \mathbb{R}^3 \to \mathbb{R}^3$ with homologically nontrivial fibers, and show that this example has $K_f \in L^{1/2 - \varepsilon}_{\mathrm{loc}}(\mathbb{R}^3)$ for all $\varepsilon > 0$.
\end{abstract}

\section{Introduction}

Let $\Omega$ and $\Omega'$ be domains in $\R^n$, $n \ge 2$. Recall that a mapping $f \colon \Omega \to \R^n$ of Sobolev class $W^{1,n}_{\loc}(\Omega, \mathbb R^n)$ has \emph{finite distortion} if
\begin{equation}\label{eq:distineq}
|Df(x)|^n \le  K(x) \,J_f(x)
\end{equation}
for some measurable function $1 \le K(x) < \infty$. Here, $\abs{Df(x)}$ stands for the operator norm of the differential $Df(x)$.
Thus, the distortion inequality~\eqref{eq:distineq} simply asks that  the Jacobian determinant $J_f(x)=\det Df(x)$ is positive at a.e.\ (almost every) point $x \in \Omega$ where $Df(x) \neq 0$. The smallest function $K(x) \ge 1$ for which the distortion inequality~\eqref{eq:distineq} holds is called the \emph{(outer) distortion function} of $f$, and is denoted by $K_f(x)$. When $K_f\in L^\infty (\Omega)$, we obtain the widely studied special case of \emph{quasiregular mappings}; see e.g.\ \cite{Iwaniec-Martin_book, Reshetnyak-book, Rickman_book}.  

In the past 20 years, there has been much systematic study of mappings of finite distortion in the field of geometric function theory (GFT).  Many of the standard results of quasiregular mappings have been proven for mappings of finite distortion with sufficient integrability assumptions on $K_f$; see e.g.\ \cite{Hencl-Koskela-book, Iwaniec-Martin_book}. The theory finds concrete applications in materials science, particularly nonlinear elasticity (NE) and critical phase phenomena, and in the calculus of variations.  

The mathematical models of NE~\cite{Antman_Elasticity-book, Ball_nonlinear-elasticity, Ciarlet_Elasticity-book}, and
a variational approach to GFT share common interests to study homeomorphisms of finite distortion and, in particular, (topologically) monotone mappings of finite distortion. Here, a mapping $f \colon X \to Y$ between topological spaces is \emph{(topologically) monotone}~\cite{Morrey_monotone} if $f$ is continuous and $f^{-1}\{y\}$ is connected for every $y \in Y$. Indeed, monotone mappings are well suited to model the \emph{weak interpenetration of matter} where, roughly speaking, squeezing of a portion of the material can occur, but not folding or tearing. In the planar setting, monotone mappings can be characterized as uniform limits of homeomorphisms by a theorem of Youngs \cite{Youngs_monotone-approx}.

To clarify our terminology, we note that in the study of mappings of finite distortion, it is also common to consider another form of monotonicity introduced by Manfredi~\cite{Manfredi_weakly_monotone}, which we call the \emph{$1$-oscillation property}. Namely, a mapping $f \colon \Omega \to \R^n$ satisfies the $1$-oscillation property if  it satisfies the estimate $\osc_B(f) \leq \osc_{\partial B}(f)$ for every ball $B \subset \Omega$, where  $\osc_K(f) = \sup_{x, x' \in K} \abs{f(x) - f(x')}$. This is a weaker definition of monotonicity, as any $W^{1,n}$-Sobolev mapping of finite distortion enjoys the $1$-oscillation property, see~\cite{Iwaniec-Koskela-Onninen_Inventiones}; this includes even maps like $z \mapsto z^2$ on the complex plane, which is clearly not topologically monotone. As another example of the difference between these definitions, folding maps which cause \emph{strong interpenetration of matter} are not topologically monotone, but may still satisfy the 1-oscillation property.

Our study is centered around the following general question: how does the integrability of $K_f$ affect the possible shapes of the fibers $f^{-1}\{y\}$, when $f \colon \Omega \onto \Omega'$ is a monotone mapping in $W_{\loc}^{1,n} (\Omega, \R^n)$.  We begin by recalling that, if a non-constant $f \in W^{1,n}_\loc(\Omega, \R^n)$ has finite distortion, $K_f \in L^{n-1}_\loc(\Omega)$, and $f$ has essentially bounded multiplicity, then $f$ is open and discrete by the main result in \cite{Hencl-Koskela_MFD-discr-open}.  Without assuming that the mapping $f$ has essentially bounded multiplicity, a slightly higher integrability for the distortion $K_f$ is required for openness and discreteness to still hold, namely $K_f \in L^{n-1+\eps}_\loc(\Omega)$ for some $\eps >0$, see~\cite{Hencl-Rajala, Manfredi-Villamor}.  However, for non-constant $W^{1,n}$-maps of finite distortion, this essential multiplicity bound always holds if $K_f \in L^{1/(n-1)}_\loc(\Omega)$, see Lemma \ref{lem:almost_every_fiber_is_trivial}. It follows that if a non-constant $f \in W^{1,n}_\loc(\Omega, \R^n)$ is monotone and $K_f \in L^{n-1}_\loc(\Omega)$, then $f$ is homeomorphic, and therefore all fibers $f^{-1}\{y\}$ with $y \in f(\Omega)$ are singletons.

The idea behind the main results of \cite{Hencl-Koskela_MFD-discr-open} is that if the required conditions are satisfied, then $\cH^1(f^{-1}\{y\}) = 0$ for every $y \in \R^n$, which is then used to show openness and discreteness of $f$. We note here that a trivial modification to the proof in \cite{Hencl-Koskela_MFD-discr-open} yields a similar result for other Hausdorff measures. Since this result serves as a starting point for our investigation, we state it here and give a few comments on the proof in Section \ref{sect:Hencl-koskela}.

\begin{thm}\label{thm:fiber_Hausdorff_bound}
	Let $\Omega \subset \R^n$ be a domain, $n \ge 2$. Suppose that $f \in W^{1,n}_\loc(\Omega, \R^n)$ is a non-constant mapping of finite distortion and the mapping $f$ has essentially bounded multiplicity. Then for $p \in \bigl[\frac{1}{n-1}, \infty\bigr)$ we have 
	\[
		K_f \in L^p_\loc(\Omega) \implies \cH^\frac{n}{p+1}(f^{-1}\{y\}) = 0 \text{ for all } y \in f(\Omega). 
	\]
\end{thm} 

\subsection{An example with looped fibers}

In \cite[Section 4]{Bing_monotone-example}, Bing gives a topological example of a monotone map $f \colon \R^3 \onto \R^3$ such that some of the fibers $f^{-1}\{y\}$ of $f$ are topologically $\S^1$ or $\S^1 \vee \S^1$. Here $\S^1 \vee \S^1$ stands for a figure-eight formed by two disjoint copies of circle $\S^1$ that have been joined at a point.
The example is part of a detailed investigation into the higher-dimensional failure of a theorem of Moore \cite{Moore-Trans25}, which states that each decomposition of $\R^2$ into continua which fail to separate $\R^2$ yields a decomposition space topologically equivalent to $\R^2$.  We refer to a book of Daverman \cite{Daverman_book} for the  development of monotone mappings as a part of the theory of decomposition spaces and manifold recognition problems.

In this paper, we construct an explicit Sobolev representation of Bing's mapping, and study its properties as a mapping of finite distortion. Here, we recall that a map $f \colon \Omega \to \Omega'$ is \emph{proper} if $f^{-1} K$ is compact for every compact $K \subset \Omega'$.

\begin{thm}\label{thm:Bing_meets_Sobolev_again}
	There exists a map $h \colon \R^3 \to \R^3$ with the following properties.
	\begin{itemize}
		\item The map $h$ is topologically monotone, proper, and surjective onto $\R^3$.
		\item The map $h$ is locally Lipschitz, and $J_h$ is positive almost everywhere. Hence, $h$ is a mapping of finite distortion.
		\item We have $K_h \in L^p_\loc(\R^3, \R^3)$ for every $p < 1/2$, but $K_h \notin L^{1/2}_\loc(\R^3, \R^3)$.
		\item The fibers $h^{-1}\{0\}$ and $h^{-1}\{-e_x\}$ are bilipschitz equivalent with $\S^1$. The fibers $h^{-1}\{-te_x\}$ for $t \in (0, 1)$ are bilipschitz equivalent with $\S^1 \vee \S^1$. The fibers $h^{-1}\{-te_x\}$ for $t \in (1, \infty)$ are bilipschitz equivalent with $[0, 1]$. For all other values $y \in \R^3 \setminus \{-te_x \colon t \geq 0\}$, the fiber $h^{-1}\{y\}$ is a point.
	\end{itemize}
\end{thm}

A notable property of this example is that adjustments to the definition of $h$ seem to fail to improve the integrability of $K_h$ past the threshold of $p = 1/2$. For comparison, the threshold imposed by Theorem \ref{thm:fiber_Hausdorff_bound} at which $1$-dimensional fibers are prevented is $p = 2$. Standard results instead imply that a monotone mapping of finite distortion $f \in W^{1,3}_\loc(\R^3, \R^3)$ with $K_f \in L^{1/2}_\loc(\R^3)$ satisfies the Lusin $(N^{-1})$-condition, and that $J_f > 0$ a.e., see \cite[Theorem 4.13]{Hencl-Koskela-book}. Here, we recall that a map $f \colon \R^3 \to \R^3$ satisfies the \emph{Lusin $(N^{-1})$-condition} if $f^{-1} A$ has zero (Lebesgue) measure for every $A \subset \R^3$ of zero measure; conversely the \emph{Lusin $(N)$-condition} is that $f(A)$ has measure zero for every $A \subset \R^3$ of measure zero.  However, the aforementioned result presents no obstruction in our case, since the map $h$ of Theorem \ref{thm:Bing_meets_Sobolev_again} does have an a.e.\ positive $J_h$, and $h$ therefore also satisfies the Lusin $(N^{-1})$-condition.

Hence, the existing results in the theory of mappings of finite distortion cannot seem to explain the apparent upper limit on the integrability of $K_h$. This suggests a potential missing result on the fact that the integrability of $K_h$ limits the possible looping of fibers. Our main goal in this paper is to prove such a result.

\subsection{Homological obstructions}

The most natural form of our main results is stated in terms of pre-images of open balls. In this setting, the statement is as follows.

\begin{thm}\label{thm:homology_obstruction}
	Let  $f \colon \Omega \to \Omega'$ be a proper, continuous, monotone surjection  in the Sobolev class $W^{1,n}_\loc(\Omega, \R^n)$, $n \ge 3$. Suppose that $k \in \{1, \dots, n-2\}$, and that
	\[
		K_f \in L^p_\loc(\Omega), \quad \text{where }
		p = \begin{cases}
				\frac{n-(k+1)}{k+1},& 1 \leq k < \frac{n}{2},\\
				1,& k = \frac{n}{2},\\
				\frac{k-1}{n-(k-1)},& \frac{n}{2} < k \leq n-2.
			\end{cases}
	\]
	Then
	\[
		H_k(f^{-1} \B^n(y, r); \R) = \{0\} \quad \text{for every } \B^n(y, r) \Subset \Omega'.
	\]
\end{thm}

Here $H_k(X; \R)$ stands for the $k$:th \emph{singular homology} group of $X$ with coefficients in $\R$, and $U \Subset V$ denotes that the closure $\overline{U}$ is a compact subset of $V$. We note that Theorem \ref{thm:homology_obstruction} does not include the cases $k = 0, n-1, n$. In the case $k = n-1$ our argument in fact does give a critical exponent $p = (n-2)/2$. However, including these cases is unnecessary, as $H_k(f^{-1} \B^n(y, r); \R)$ for $k \in \{0, n-1, n\}$ is determined entirely by the topological properties of $f$ by the following standard result.

\begin{prop}\label{prop:topo_preimage_restrictions}
	Let $f \colon \Omega \to \Omega'$ be a proper, continuous, monotone surjection between open domains in $\R^n$. Then for every $\B^n(y, r) \Subset \Omega'$, we have 
	\begin{align*}
		H_0(f^{-1} \B^n(y, r); \R) &\cong \R,\\
		H_{n-1}(f^{-1} \B^n(y, r); \R) &\cong \{0\},\\
		H_{n}(f^{-1} \B^n(y, r); \R) &\cong \{0\}.
	\end{align*}
\end{prop}

We then consider a fiber $f^{-1}\{y\}$ of a monotone map $f$ satisfying the assumptions of Theorem \ref{thm:homology_obstruction}. The sets $f^{-1} B^n(y, i^{-1})$ for large enough $i \in \Z_+$ form a descending sequence of precompact neighborhoods of $f^{-1}\{y\}$, and the intersection of these neighborhoods is $f^{-1}\{y\}$. We are now interested in whether the triviality of the sets $H_k(f^{-1} B^n(y, i^{-1}); \R)$ implies the triviality of $H_k(f^{-1}\{y\}; \R)$.

One example of a situation in which this does occur is if $f^{-1} \{y\}$ is a \emph{neighborhood retract}; that is, if there exists a neighborhood $U \subset \R^n$ of $f^{-1} \{y\}$ and a retraction $r \colon U \to f^{-1} \{y\}$. One class of examples of neighborhood retracts are closed manifolds with a tubular neighborhood, which for example include all embedded smooth closed submanifolds of $\R^n$.

\begin{cor}\label{cor:fiber_homology_obstruction}
	Let  $f \colon \Omega \to \Omega'$ be a proper, continuous, monotone surjection  in the Sobolev class $W^{1,n}_\loc(\Omega, \R^n)$, $n \ge 3$.  Let $k \in \{1, \dots, n\}$. Moreover, if $k \leq n-2$, suppose also that
	\[
		K_f \in L^p_\loc(\Omega), \quad \text{where }
		p = \begin{cases}
			\frac{n-(k+1)}{k+1},& 1 \leq k < \frac{n}{2},\\
			1,& k = \frac{n}{2},\\
			\frac{k-1}{n-(k-1)},& \frac{n}{2} < k \leq n-2.
		\end{cases}
	\]
	If $y \in \Omega'$ is such that $f^{-1}\{y\}$ is a neighborhood retract, then $H_k(f^{-1} \{y\}; \R) = \{0\}$.
\end{cor}

For an arbitrary compact connected set $K \subset \R^n$, it is possible that $H_k(K; \R) \neq \{0\}$ even if $H_k(U_i; \R) = \{0\}$ for a decreasing sequence of pre-compact neighborhoods $U_i$ of $K$ with $K = \bigcap_i U_i$. For an example of this, consider
\[
	K = \overline{\left\{(x, y, z) \in \R^3 : 0 < x^2 + y^2 \leq 1, z = \sin\left( \pi/\sqrt{x^2 + y^2}\right)\right\}}.
\]
That is, $K$ is the closed topologist's sine curve that has been revolved around the $z$-axis. The set $K$ is compact and connected, though it is not path connected. Moreover, the loop $S = \{(x, y, z) \in \R^3 : z = 0, x^2 + y^2 = 1\} \subset K$ induces a non-zero homology class in $H_1(K; \R)$, but $S$ is homologically trivial in any neighborhood of $K$. It is, however, unknown to us whether any examples similar to $K$ can occur as a fiber of a monotone $W^{1,n}$-map of finite distortion $f$, and if yes, what restrictions this would place on the degree of integrability of $K_f$.

The map $h$ of Theorem \ref{thm:Bing_meets_Sobolev_again} shows that Theorem \ref{thm:homology_obstruction} and Corollary \ref{cor:fiber_homology_obstruction} are sharp when $n = 3$, $k = 1$. In this way, our results explain the difficulties in trying to improve the integrability of $K_f$ beyond $p = 1/2$. It is unknown to us whether these bounds are sharp for other values of $n, k$. See Figure \ref{fig:p_values} for a table of the critical exponents $p$.

\begin{figure}[h]
	\begin{tikzpicture}[yscale=0.6]
		\node at (-6, -1) {$n = 3$};
		\node at (0-1/2,-1) {$\frac{1}{2}$};
		\node at (1-1/2,-1) {$(\frac{1}{2})$};
		
		\node at (-6, -2) {$n = 4$};
		\node at (0-2/2,-2) {$1$};
		\node at (1-2/2,-2) {$1$};
		\node at (2-2/2,-2) {$(1)$};
		
		\node at (-6, -3) {$n = 5$};
		\node at (0-3/2,-3) {$\frac{3}{2}$};
		\node at (1-3/2,-3) {$\frac{2}{3}$};
		\node at (2-3/2,-3) {$\frac{2}{3}$};
		\node at (3-3/2,-3) {$(\frac{3}{2})$};
		
		\node at (-6, -4) {$n = 6$};
		\node at (0-4/2,-4) {$2$};
		\node at (1-4/2,-4) {$1$};
		\node at (2-4/2,-4) {$1$};
		\node at (3-4/2,-4) {$1$};
		\node at (4-4/2,-4) {$(2)$};
		
		\node at (-6, -5) {$n = 7$};
		\node at (0-5/2,-5) {$\frac{5}{2}$};
		\node at (1-5/2,-5) {$\frac{4}{3}$};
		\node at (2-5/2,-5) {$\frac{3}{4}$};
		\node at (3-5/2,-5) {$\frac{3}{4}$};
		\node at (4-5/2,-5) {$\frac{4}{3}$};
		\node at (5-5/2,-5) {$(\frac{5}{2})$};
		
		\node at (-6, -6) {$n = 8$};
		\node at (0-6/2,-6) {$3$};
		\node at (1-6/2,-6) {$\frac{5}{3}$};
		\node at (2-6/2,-6) {$1$};
		\node at (3-6/2,-6) {$1$};
		\node at (4-6/2,-6) {$1$};
		\node at (5-6/2,-6) {$\frac{5}{3}$};
		\node at (6-6/2,-6) {$(3)$};
	\end{tikzpicture}

	\caption{\small Values of $p$ in Theorem \ref{thm:homology_obstruction} as $k = 1, \dots, n-2$. The unnecessary case $k = n-1$ is also listed in parenthesis to make the diagram symmetric.}\label{fig:p_values}
\end{figure}

\subsection{Connections to homeomorphic approximation}

Part of our motivation in studying this topic lies in questions related to approximating maps by homeomorphisms. For instance, such a  question  raises in the context of neohookean materials.
The \emph{neohookean material},  defined based on Hooke's law,  refers to a stored energy function  which increases to infinity when the Jacobian determinant $J_f$ approaches zero, see e.g.~\cite{Ball-global, Bauman-Owen-Phillips,  Fonseca-Gangbo,  Ciarlet-Necas, Conti-DeLellis,  Muller-Qi-Yan, Muller-Spector, Muller-Spector-Tang, Sverak}. The model examples
take the form
\begin{equation}\label{eq:neo}
\mathsf E_q^p [f] = \int_\Omega \left[ \abs{Df}^p+ J_f^{-q} \right]\, \dtext x \, ,  \quad p \ge n\, \quad    q>0 \quad \textnormal{and} \quad  \Omega \subset \R^n \, . 
\end{equation}
This model is also broadly studied by  physicists, materials scientists and engineers~\cite{Treloar_book}.  

Nonetheless, establishing non-interpenetration of matter in this setting remains a mathematical challenge. Naturally the first step towards to understanding the injectivity  of  minimizers is to enlarge the class of admissible homeomorphisms. Adopting the class of monotone maps of finite distortion ensures the existence of minimizers. However, to show that there is no \emph{Lavrentiev gap} between the classes of homeomorphisms and monotone maps leads to a suitable approximation question. Before proceeding to illuminate the general problem of approximating a monotone map by homeomorphisms, we note that a mapping $f \colon \Omega \to \R^n$ with $\mathsf E_q^p [f]  < \infty$ has finite distortion $K_f \in L^r_{\loc} (\Omega)$ where $n/p + 1/q = 1/r$.

Consider a continuous map $f \colon \overline{\Omega} \to \overline{\Omega'}$ between two simply connected planar Jordan domains $\Omega, \Omega' \subset \C$. By the theorem of Youngs \cite{Youngs_monotone-approx}, $f$ can be approximated uniformly with homeomorphisms if and only if $f$ is monotone. If moreover $f \in W^{1,p}(\Omega, \Omega')$ and the domain $\Omega'$ is Lipschitz regular, then a uniform homeomorphic approximation of $f$ can be improved to also converge in the $W^{1,p}$-norm, $1<p<\infty$, see \cite{Iwaniec-Onninen_Monotone-approx}.

Consider then a similar situation in higher dimensions. We restrict ourselves here to the simple case where $f \colon \overline{\B^n} \to \overline{\B^n}$ is continuous, $f(\B^n) = \B^n$, and $f \colon \partial \B^n \to \partial \B^n$ is a homeomorphism. Under which conditions can the map $f$ be uniformly approximated by homeomorphisms $f_i \colon \overline{\B^n} \to \overline{\B^n}$?

In this higher-dimensional case, monotonicity of $f$ is no longer sufficient. One additional necessary condition is that every fiber $f^{-1}\{y\}$ is \emph{cellular}; that is, $f^{-1}\{y\}$ is an intersection of a nested sequence of topological balls. Indeed, if $f_i \colon \overline{\B^n} \to \overline{\B^n}$ is a homeomorphism, and $\abs{f - f_i} < \eps/2$ uniformly, then for any $y \in \B^n$ with $d(y, \partial \B^n) > 2\eps$, we have that $U_\eps = f_i^{-1} \B^n(y, \eps)$ is a topological ball such that $f^{-1} \B^n(y, \eps/2) \subset U_\eps \subset f^{-1} \B^n(y, 2\eps)$. Note that such an approximation is impossible if $f^{-1}\{y\}$ is e.g.\ a smoothly embedded copy of $\S^1$; in particular,  maps similar to the one discussed in Theorem \ref{thm:Bing_meets_Sobolev_again} can not be homeomorphically approximated.

It turns out that for the most part, this extra necessary condition of cellular fibers is sufficient for uniform approximation. Indeed, a result of Siebenmann \cite{Siebenmann_cellular-approx} yields that if our map $f$ is monotone, $f$ has cellular fibers, and $n \neq 4$, then $f$ can be uniformly approximated by homeomorphisms. We note that the result is formulated in terms of a more general definition of CE-maps, which in particular holds for continuous proper $f \colon \B^n \to \B^n$ with cellular fibers and homeomorphic boundary values.

Now, we consider the situation where $f$ is a mapping of finite distortion. Hence, we obtain the following result from Theorem \ref{thm:homology_obstruction} and Proposition \ref{prop:topo_preimage_restrictions}.

\begin{cor}\label{cor:homologically_cellular_fibers}
	Let $f\colon \Omega \to \Omega' $ be a proper, continuous, monotone surjection in $\W^{1,n}(\Omega, \R^n)$, $n \ge 3$. Suppose that $K_f \in L^{(n-2)/2}_\loc(\Omega)$. Then for every $y \in \Omega'$, the set $f^{-1}\{y\}$ is an intersection of a nested sequence of neighborhoods $U_i$ that are rational homology balls; that is, the neighborhoods $U_i$ satisfy $H_k(U_i; \Q) = H_k(\B^n; \Q)$ for all $k \in \Z_{\geq 0}$.
\end{cor}

That is, under the assumptions of Corollary \ref{cor:homologically_cellular_fibers}, while we do not know if the fibers $f^{-1}\{y\}$ are an intersection of topological balls and therefore cellular, we do know that the fibers are intersections of neighborhoods that look like topological balls through the lens of rational homology. Note that the condition we need to assume corresponds by Theorem \ref{thm:fiber_Hausdorff_bound} to the fibers having zero $\cH^2$-measure. 
For an example of a rational homology ball $U$ that is not homeomorphic to $\B^3$, consider e.g.\ any homeomorphic copy of $U = \S^3 \setminus H$ where $H$ is a filled-in Alexander horned ball; see e.g.\ \cite[p. 169]{Hatcher_AlgTopo}. 

Hence, our results suggest the following question on homeomorphic approximation in three dimensions. Similar questions can also be stated in higher dimensions, but the case $n = 3$ is especially notable since the example of Theorem \ref{thm:Bing_meets_Sobolev_again} essentially shows that the assumptions cannot be improved.

\begin{qu}
Let  $f \colon \B^3 \onto \B^3 $ be a continuous, proper, monotone mapping of finite distortion in $W^{1,3}(\B^3, \R^3)$. Suppose that $K_f \in L^{1/2}(\B^3)$, and that $f$ extends to a continuous $f \colon \overline{\B^3} \to \overline{\B^3}$ with homeomorphic boundary values. Can $f$ be uniformly approximated by homeomorphisms $f_i \colon \overline{\B^3} \to \overline{\B^3}$? If yes, can this approximation be improved to a uniform and $W^{1,3}$ -approximation by $W^{1,3}$-homeomorphisms?
\end{qu}

\subsection{Idea of the proofs}

Our strategy in showing Theorem \ref{thm:homology_obstruction} starts with ideas from the study of the discreteness and openness of mappings of finite distortion, and then combines these ideas with the use of Sobolev de Rham cohomology theories. In particular, the main idea of the proofs can be essentially condensed to a single diagram:
\[\begin{tikzcd}
	C^\infty(\wedge^k f^{-1} \B^n(y, r)) \ar[r, "\push{f}"] & 
	L^\frac{n}{k}_\loc(\wedge^k \B^n(y, r)) 
		\ar[d, "\text{Sobolev-Poincar\'e}"]\\
	L^1_\loc(\wedge^{k-1} f^{-1} \B^n(y, r)) &
	L^\frac{n}{k-1}_\loc(\wedge^{k-1} \B^n(y, r)) 
	\ar[l, "f^*"]
\end{tikzcd}\]

Indeed, we take a smooth closed $k$-form $\omega$ on $f^{-1} \B^n(y, r)$, and push it forward in $f$ to a form $\push{f}\omega$ on $\B^n(y, r)$. Since the conformal $k$-cohomology of $\B^n(y, r)$ is trivial, we have $\push{f}\omega = d\tau$, where we may assume that the $(k-1)$-form $\tau$ is $L^{n/(k-1)}$-integrable by the Sobolev-Poincar\'e inequality. It then follows that $\omega = d f^* \tau$, implying that $\omega$ is trivial in local $L^1$-cohomology.

In order for the push-forward map to have the correct target space, we need that $K_f \in L^{(n-k)/k}_\loc(\Omega)$. For the pull-back, we similarly need that $K_f \in L^{(k-1)/(n-k-1)}_\loc(\Omega)$. Hence, under these assumptions, the above computation and a de Rham theorem for $L^p_\loc$-cohomologies lets us deduce that $H^k(f^{-1} \B^n(y, r); \R)$ vanishes. The cases $k \geq n/2$ in Theorem \ref{thm:homology_obstruction} hence follow by the universal coefficient theorem. Note that the use of the conformal exponent on the image side is crucial for our argument, as a higher intermediary exponent will break the push-forward map, and a lower one will similarly break the pull-back map.

For the cases $k \leq n/2$ in Theorem \ref{thm:homology_obstruction}, we use compactly supported cohomology. Indeed, since $f^{-1} \B^n(y, r)$ is an open subset of $\R^n$, its $k$-homology spaces are isomorphic to its compactly supported $(n-k)$-cohomology spaces by Poincar\'e duality. We can hence replace $k$ with $n-k$ and perform the same argument with spaces of compactly supported forms, which yields our result in the cases $k \leq n/2$.

\section{Differential forms and maps of finite distortion}

We consider a continuous, proper, surjective, topologically monotone $W^{1,n}$-map $f \colon \Omega \to \Omega'$ between open domains in $\R^n$. We begin by recalling the following useful facts about such maps.

\begin{lemma}\label{lem:almost_every_fiber_is_trivial}
	Let $f \in W^{1,n}(\Omega, \Omega')$ be a non-constant continuous monotone map between open domains. Suppose that $f$ is a mapping of finite distortion with $K_f^{1/(n-1)} \in L^1(\Omega)$. Then $f^{-1}\{y\}$ is a singleton for a.e.\ $y \in f(\Omega)$, $f$ satisfies both the Lusin $(N)$ and $(N^{-1})$-conditions, and $J_f > 0$ almost everywhere.
\end{lemma}
\begin{proof}
	We note that $f$ satisfies the Lusin (N) -condition by \cite[Theorem 4.5]{Hencl-Koskela-book}, and that $f$ is differentiable almost everywhere by \cite[Corollary 2.25 b)]{Hencl-Koskela-book}. 
	
	Let $B$ be the set of points $x \in \Omega$ where $f^{-1}(f(x))$ is not a singleton. If $x \in B$, then due to the monotonicity of $f$, we find a sequence of points $x_j \to x$ such that $f(x_j) = f(x)$ for every $j \in \Z_+$. If $f$ is also differentiable at $x$, then we must have $Df(x)v = 0$ for some vector $v \in \S^{n-1}$. It follows that $J_f(x) = 0$ at every such $x$. 
	
	We thus obtain that $J_f \equiv 0$ almost everywhere in $B$. Due to $f$ satisfying the Lusin (N) -condition, we may hence use change of variables to conclude that
	\[
		m_n(f(B)) \leq \int_B J_f = 0,
	\]
	which completes the proof of the fact that $f^{-1}\{y\}$ is a singleton for a.e.\ $y \in f(\Omega)$. The Lusin $(N^{-1})$ -condition and the fact that $J_f > 0$ a.e.\ now follow from \cite[Theorem 4.13]{Hencl-Koskela-book} using the assumption $K_f^{1/(n-1)} \in L^1(\Omega)$, since the multiplicity function of $f$ is essentially bounded from above by 1.
\end{proof}

We then consider how the integrability of $K_f$ effects the pull-backs of differential forms by $f$. In the following lemma we are mainly interested in the case $p = n/k$, but we regardless give a more general statement.

\begin{lemma}\label{lem:pullback_estimate}
	Let $f \in W^{1,n}(\Omega, \Omega')$ be a continuous, proper, monotone surjection, where $\Omega, \Omega'$ are open domains. Let $\omega \in L^p(\wedge^k \Omega')$ and let $K_f^q \in L^1(\Omega)$, with $k \in \{1, \dots, n\}$, $n/k \leq p \leq \infty$ and $(n-1)^{-1} \leq q \leq \infty$. Then
	\[
		\abs{f^* \omega}^r \in L^1(\Omega) \qquad \text{where } r = \frac{n}{k + \frac{n}{pq}}.
	\]
	More precisely, we have the estimate
	\[
		\norm{f^* \omega}_r \leq \norm{\omega}_p \norm{K_f}_q^\frac{1}{p} \norm{Df}_n^{\frac{qk-n}{q}}.
	\]
\end{lemma}
\begin{proof}
	The measurability of $f^* \omega$ follows from the Lusin $(N^{-1})$-condition. The case $q = \infty$ is an immediate consequence of the standard result that quasiregular maps preserve the $(n/k)$-integrability of $k$-forms; see e.g.\ \cite[Section 2.2]{Kangasniemi-Pankka_PLMS}. The case $p = \infty$ is similarly simple, as we can then estimate
	\[
		\int_\Omega \abs{f^* \omega}^\frac{n}{k} \leq 
		\int_\Omega (\abs{\omega}^\frac{n}{k} \circ f) \abs{Df}^{n}
		\leq \norm{\omega}_\infty^\frac{n}{k} \int_\Omega \abs{Df}^{n} < \infty.
	\]
	
	We then consider the case $p \neq \infty \neq q$, where we use H\"older's inequality to estimate that
	\begin{multline*}
		\int_\Omega \abs{f^* \omega}^r \leq 
		\int_\Omega (\abs{\omega}^r \circ f) \abs{Df}^{rk}
		= \int_\Omega (\abs{\omega}^r \circ f) J_f^{\frac{r}{p}} K_f^{\frac{r}{p}} \abs{Df}^{(k - p^{-1} n)r}\\
		\leq \left( \int_\Omega (\abs{\omega}^p \circ f) J_f 
			\right)^{\frac{r}{p}}
		\left( \int_\Omega K_f^q 
			\right)^{\frac{r}{pq}}
		\left( \int_\Omega \abs{Df}^{\frac{(pqk-qn)r}{pq - qr - r}}
			\right)^{\frac{pq - qr - r}{pq}}.
	\end{multline*}
	Note that our use of H\"older is valid if $	pq - qr - r \geq 0$, which holds since
	\begin{align*}
		pq - qr - r = \frac{q(pk-n)}{k + \frac{n}{pq}}
	\end{align*}
	and since we assumed that $p \geq n/k$. With a change of variables, we have
	\[
		\int_\Omega (\abs{\omega}^p \circ f) J_f = \int_{\Omega'} \abs{\omega}^p < \infty.
	\]
	Finally, we see using our definition of $r$ that
	\[
		\frac{(pqk-qn)r}{pq - qr - r} = \frac{(pqk-qn)n}{pq(k + \frac{n}{pq}) - qn - n} = \frac{(pqk-qn)n}{kpq + n - nq - n} = n,
	\]
	and hence
	\[
		\int_\Omega \abs{Df}^{\frac{(pqk-qn)r}{pq - qr - r}}
		= \int_\Omega \abs{Df}^n < \infty.
	\]
\end{proof}

Next, we wish to define a push-forward map for a continuous monotone surjection $f \in W^{1,n}(\Omega, \Omega')$ with integrable $K_f^{1/(n-1)}$. By Lemma \ref{lem:almost_every_fiber_is_trivial}, for a.e.\ $y \in \Omega'$ there exists a unique point $f^{-1}(y) \in \Omega$ such that $f(f^{-1}(y)) = y$. Given a differential $k$-form $\omega$ on $\Omega$, we define
\begin{equation}\label{eq:push_def}
	(\push{f} \omega)_y = \omega_{f^{-1}(y)} \circ \wedge^k [Df(f^{-1}(y))]^{-1}
\end{equation}
for a.e.\ $y \in \Omega'$. Using the Lusin conditions of $f$, it can be seen that the resulting map is measurable for measurable forms $\omega$.

We then prove a similar norm estimate for the push-forward map.
\begin{lemma}\label{lem:pushforward_estimate}
	Let $f \in W^{1,n}(\Omega, \Omega')$ be a continuous, proper, monotone surjection, where $\Omega, \Omega'$ are open domains. Let $\omega \in L^p(\wedge^k \Omega)$ with $k \in \{1, \dots, n\}$ and $n/k \leq p \leq \infty$. Then
	\[
		\abs{\push{f} \omega} \in L^\frac{n}{k}(\Omega') \quad \text{if} \quad K_f \in L^\frac{(n-k)p}{kp-n}(\Omega).
	\]
	More precisely, we have the estimate
	\[
		\norm{\push{f} \omega}_\frac{n}{k} \leq \norm{\omega}_{p} \norm{K_f}_{\frac{(n-k)p}{kp-n}}^\frac{n-k}{n}.
	\]
\end{lemma}
\begin{proof}
	Note that $(n-k)p/(kp-n) \geq (n-k)/k \geq 1/(n-1)$, so the push-forward is well defined. We note that for a.e.\ $x \in \Omega$ and all $v \in \wedge^k T_x \Omega$, we have
	\[
		\ip{\wedge^k Df(x) v}{\hodge (\wedge^{n-k} Df(x)) \hodge v'}
		= \hodge (\wedge^n Df(x)) (v \wedge \hodge v')
		= \ip{v}{v'} J_f(x).
	\]
	Applying this with $v = \wedge^k [Df(x)]^{-1} w$, $w \in T_{f(x)} \Omega'$, and $\abs{w} = \abs{v'} = 1$, leads to the estimate
	\[
		\abs{\wedge^k [Df(x)]^{-1}} \leq J_f(x)^{-1} \abs{\wedge^{n-k} Df(x)} \leq J_f(x)^{-1} \abs{Df(x)}^{n-k}.
	\]
	Hence, by using the almost everywhere defined function $f^{-1}$, a change of variables in $f$ gives us
	\begin{multline*}
		\int_{\Omega'} \abs{\push{f} \omega}^\frac{n}{k}
		\leq \int_{\Omega'} \left(\abs{\omega} J_f(x)^{-1}\abs{Df(x)}^{n-k}\right)^\frac{n}{k} \circ f^{-1}\\
		= \int_{\Omega} \abs{\omega}^\frac{n}{k} J_f^{-\frac{n-k}{k}} \abs{Df(x)}^{\frac{(n-k)n}{k}}
		= \int_{\Omega} \abs{\omega}^\frac{n}{k} K_f^\frac{n-k}{k}f.
	\end{multline*}
	In the case $p = \infty$, we may estimate $\abs{\omega} \leq \norm{\omega}_\infty$, and the result follows since our assumed degree of integrability from $K_f$ is exactly $(n-k)/k$ in this case. In other cases, we use H\"older's inequality, and get the desired
	\[
		\int_{\Omega} \abs{\omega}^\frac{n}{k} K_f^\frac{n-k}{k}f
		\leq \left( \int_{\Omega} \abs{\omega}^p \right)^{\frac{n}{pk}}
		\left( \int_{\Omega} K_f^{\frac{n-k}{k} \cdot \frac{pk}{pk-n}} \right)^{\frac{pk-n}{pk}} < \infty.
	\]
\end{proof}

\subsection{Weak differentials}

We let $W^{d, p, q}(\wedge^k \Omega)$ denote the space of measurable differential $k$-forms $\omega \in L^p(\wedge^k \Omega)$ which have a weak differential $d\omega \in L^q(\wedge^{k+1} \Omega)$. Recall that a $(k+1)$-form $d\omega \in L^1_\loc(\wedge^{k+1} \Omega)$ is a weak differential of $\omega \in L^1_\loc(\wedge^k \Omega)$ if
\[
	\int_\Omega \omega \wedge d\eta = (-1)^{k+1} \int_\Omega d\omega \wedge \eta
\]
for every compactly supported smooth $\eta \in C^\infty_c(\wedge^{n-k-1} \Omega)$. We use the shorthand $W^{d,p}(\wedge^k \Omega) = W^{d,p,p}(\wedge^k \Omega)$.

We similarly use $W^{d, p, q}_\loc(\wedge^k \Omega)$ to denote the space of measurable $k$-forms $\omega \in L^p_\loc(\wedge^k \Omega)$ with a weak differential $d\omega \in L^q_\loc(\wedge^{k+1} \Omega)$. We also denote by $W^{d,p,q}_c(\wedge^k \Omega)$ the space of compactly supported elements of $W^{d, p, q}(\wedge^k \Omega)$; recall that the \emph{support} $\spt \omega$ of $\omega \in L^p_\loc(\wedge^k \Omega)$ is the set of all $x \in \Omega$ such that there exists no neighborhood $U$ of $x$ with $\omega = 0$ a.e.\ on $U$. Similarly as above, we use the shorthands $W^{d,p}_\loc(\wedge^k \Omega) = W^{d,p,p}_\loc(\wedge^k \Omega)$ and $W^{d,p}_c(\wedge^k \Omega) = W^{d,p,p}_c(\wedge^k \Omega)$.

We recall the following standard result which implies that $f^* d\omega = d f^* \omega$ when $\omega$ is smooth and $f$ is a continuous $W^{1,p}_\loc$-map with suitably high $p$. For the case when $\omega$ is compactly supported, we refer to e.g. \cite[Lemma 2.2]{Kangasniemi-Onninen_Heterogeneous}, and the general version follows using the continuity of $f$ and a locally finite partition of unity. 

\begin{lemma}\label{lem:pullback_d_commutation}
	Let $\Omega, \Omega' \subset \R^n$ be open domains. Suppose that $f \in C(\Omega, \Omega') \cap W^{1,p}_\loc(\Omega, \Omega')$. If $\omega \in C^\infty(\wedge^k M)$ and $p \geq k + 1$, then $f^* \omega \in W^{d, p/k, p/(k+1)}_\loc(\wedge^k \Omega)$ and $d f^* \omega = f^* d \omega$.
\end{lemma}

Using Lemma \ref{lem:pullback_d_commutation}, we prove a similar result for the push-forward in our setting.

\begin{lemma}\label{lem:pushforward_d_commutation}
	Let $f \in W^{1,n}(\Omega, \Omega')$ be a continuous, proper, monotone surjection, where $\Omega, \Omega' \subset \R^n$ are open, bounded domains. Let $\omega \in C^{\infty}(\wedge^k \Omega)$ with $k \in \{1, \dots, n-1\}$, and let $K_f \in L^{(n-k)/k}(\Omega)$. Then we have $\push{f} \omega \in W^{d, n/k, n/(k+1)}(\wedge^k \Omega')$ and $d \push{f} \omega = \push{f} d \omega$.
\end{lemma}
\begin{proof}
	Due to Lemma \ref{lem:pushforward_estimate}, the only thing we have to check is that $d \push{f} \omega = \push{f} d \omega$ in the weak sense. Let $\eta \in C^\infty_c(\wedge^{n-k-1} \Omega')$. We use a Sobolev change of variables to conclude that
	\[
		\int_{\Omega'} (\push{f} d\omega) \wedge \eta
		= \int_{\Omega} f^*((\push{f} d\omega) \wedge \eta)
		= \int_{\Omega} d\omega \wedge f^* \eta.
	\]
	By Lemma \ref{lem:pullback_d_commutation}, $f^* \eta \in W^{d,n/(n-k-1), n/(n-k)}_\loc(\wedge^{n-k-1} \Omega)$ and $f^* d\eta = d f^* \eta$. Moreover, since $f$ is proper, the form $f^* \eta$ is compactly supported. We let $\alpha_j$ be the convolutions of $f^* \eta$ with a sequence of mollifying kernels, in which case $\alpha_j \in C^\infty_c(\wedge^{n-k-1} \Omega)$ for large enough $j$, $\alpha_j \to f^*\eta$ in the $L^{n/(n - k - 1)}$-norm, and $d\alpha_j \to f^*d\eta$ in the $L^{n/(n - k)}$-norm. A standard application of H\"older's inequality for wedge products now implies that
	\begin{multline*}
		\int_{\Omega} d\omega \wedge f^* \eta
		= \lim_{i \to \infty} \int_{\Omega} d\omega \wedge \alpha_i
		= (-1)^{k+1} \lim_{i \to \infty} \int_{\Omega} \omega \wedge d \alpha_i\\
		= (-1)^{k+1} \int_{\Omega} \omega \wedge f^*d\eta.
	\end{multline*}
	Finally, one more change of variables gives us our result by
	\[
		\int_{\Omega} \omega \wedge f^*d\eta
		= \int_{\Omega} f^*(\push{f} \omega \wedge d\eta)
		= \int_{\Omega'} \push{f} \omega \wedge d\eta.
	\]
\end{proof}

We also require a somewhat specific commutation result for the weak exterior derivative $d$ and the pull-back map. The proof is similar to the proof of the previous lemma.

\begin{lemma}\label{lem:pullback_commutation_lemma}
	Let $f \in W^{1,n}(\Omega, \Omega')$ be a continuous, proper, monotone surjection, where $\Omega, \Omega' \subset \R^n$ are open domains. Let $\omega \in C^{\infty}(\wedge^k \Omega)$ with $k \in \{2, \dots, n\}$, and let $K_f^q \in L^1(\Omega)$ with $(n-1)^{-1} \leq q \leq \infty$. Suppose that $\push{f} \omega = d\tau$ weakly for some $\tau \in L^{n/(k-1)}_\loc(\wedge^{k-1} \Omega')$. If
	\[
		n \geq (k-1)(1+q^{-1}),
	\]
	then $\omega = d f^* \tau$ weakly.
\end{lemma}
\begin{proof}
	By Lemma \ref{lem:pullback_estimate}, our assumption that $n \geq (k-1)(1+q^{-1})$ implies that $f^*\tau \in L^1_\loc(\wedge^{k-1} \Omega)$. Hence, the remaining check is again that $d f^* \tau = \omega$ weakly. Let $\eta \in C^\infty_c(\wedge^{n - k} \Omega)$. A Sobolev change of variables gives us
	\[
		\int_\Omega \omega \wedge \eta
		= \int_{\Omega} f^* (\push{f}\omega \wedge \push{f}\eta)
		= \int_{\Omega'} d\tau \wedge \push{f}\eta.
	\]
	By Lemmas \ref{lem:pushforward_d_commutation} and \ref{lem:pushforward_estimate}, we have that $\push{f} \eta \in W^{d, n/(n-k), n/(n-k+1)}_c(\wedge^{n-k} \Omega')$ and $\push{f} d\eta = d \push{f} \eta$. Hence, we may again take a sequence of mollifications $\alpha_j$ of $\push{f} d\eta$, and we get that $\alpha_j \in C^\infty_c(\wedge^{n-k} \Omega')$ for large enough $j$, $\alpha_j \to \push{f} \eta$ in the $L^{n/(n-k)}$-norm, and $d\alpha_j \to \push{f} d\eta$ in the $L^{n/(n-k+1)}$-norm. Yet again a standard application of H\"older's inequality yields
	\begin{multline*}
		\int_{\Omega'} d\tau \wedge \push{f}\eta
		= \lim_{i \to \infty} \int_{\Omega'} d\tau \wedge \alpha_i
		= (-1)^{k+1}\lim_{i \to \infty} \int_{\Omega'} \tau \wedge d\alpha_i\\
		= (-1)^{k+1}\int_{\Omega'} \tau \wedge \push{f} d\eta.
	\end{multline*}
	Finally, we perform one more change of variables:
	\[
		\int_{\Omega'} \tau \wedge \push{f} d\eta
		= \int_{\Omega} f^* (\tau \wedge \push{f} d\eta)
		= \int_{\Omega} f^* \tau \wedge d\eta.
	\]
\end{proof}

\section{Sobolev de Rham cohomologies}

Let $M$ be an oriented Riemannian manifold without boundary. Note that for the purposes of this text, we only need the following results when $M$ is an open domain in $\R^n$, but we state them more generally regardless. We use similar notation $W^{d, p, q}(\wedge^k M)$, $W^{d,p}(\wedge^k M)$, $W^{d,p,p}_\loc(\wedge^k M)$, $W^{d,p}_\loc(\wedge^k M)$, $W^{d,p,p}_c(\wedge^k M)$ and $W^{d,p}_c(\wedge^k M)$ for manifolds as we specified in the Euclidean setting.

We begin by recalling a \emph{Poincar\'e lemma} in the Sobolev setting. It is closely tied with the \emph{Sobolev-Poincar\'e inequalities} for differential forms. We give here the precise version we require. Note that the following result is very close to the one stated in \cite[Lemma 4.2]{Kangasniemi-Pankka_PLMS}, but our statement also includes the exceptional $L^1$-case since we need it here.

\begin{lemma}\label{lem:poincare_lemma}
	Let $M$ be an oriented Riemannian $n$-manifold without boundary with $n \geq 2$, let $k \in \{1, \dots, n\}$, let $x \in M$, and let $\omega \in L^q_\loc(\wedge^k M)$ for some $q \in [1, \infty)$. Suppose that $d\omega = 0$ weakly. Then there exists a neighborhood $U$ of $x$ and a $(k-1)$-form $\tau \in W^{d,q}_\loc(\wedge^{k-1} U)$ such that $\omega \vert_U = d\tau$. Moreover, if $q > 1$, then we also have $\tau \in L^p_\loc(\wedge^{k-1} U)$ for every $p \in [1, \infty)$ satisfying $p^{-1} + n^{-1} \geq q^{-1}$.
\end{lemma}
\begin{proof}
	By restricting to a small enough neighborhood of $x$ and using a smooth bilipschitz chart, we may assume that $M$ is a convex Euclidean domain. Let $U$ be a small ball around $x$. 
	
	We then refer to \cite{Iwaniec-Lutoborski}, where an integral operator $T$ is constructed that is bounded $L^q(\wedge^k U) \to L^q(\wedge^{k-1} U)$ when $1 \leq q < \infty$, and that satisfies the chain homotopy condition $\alpha = T(d\alpha) + dT(\alpha)$ for all $\alpha \in W^{d,q}_\loc(\wedge^k U)$; see \cite[(4.15--4.16)]{Iwaniec-Lutoborski}. Due to $d\omega = 0$ and the chain homotopy condition, we have $dT(\omega) = \omega$, and due to the boundedness of $T$, we have $T(\omega) \in W^{d,q}(\wedge^{k-1} U)$.
	
	Moreover, when $q > 1$, it is shown in \cite[Proposition 4.1]{Iwaniec-Lutoborski} that the $W^{1,q}$-norm of $T(\omega)$ is controlled by the $L^q$-norm of $\omega$. Hence, the Sobolev embedding theorem also implies that $T(\omega) \in L^p(\wedge^{k-1} U)$ whenever $1 \leq p < \infty$ and $p^{-1} + n^{-1} \geq q^{-1}$, completing the proof.
\end{proof}

The \emph{$L^p_\loc$-cohomology} $H_{p}^*(M)$ of $M$ is the cohomology of the chain complex
\[
	0 \rightarrow W^{d,p}_\loc(\wedge^0 M) \xrightarrow{d} W^{d,p}_\loc(\wedge^1 M) \xrightarrow{d} \dots.
\]
That is, for every $k \in \{0, 1, \dots\}$, $H_{p}^k(M)$ is the quotient vector space
\[
	H_p^k(M) = \dfrac{\ker(d \colon W^{d,p}_\loc(\wedge^k M) \to W^{d,p}_\loc(\wedge^{k+1} M))}{\im (d \colon W^{d,p}_\loc(\wedge^{k-1} M) \to W^{d,p}_\loc(\wedge^{k} M))}.
\]
Similarly, we define the \emph{compactly supported $L^p$-cohomology} $H_{p, c}^*(M)$ of $M$ as the cohomology of the chain complex
\[
	0 \rightarrow W^{d,p}_c(\wedge^0 M) \xrightarrow{d} W^{d,p}_c(\wedge^1 M) \xrightarrow{d} \dots.
\]

We then require the following standard result on equivalence of cohomologies.

\begin{thm}\label{thm:Lp_cohomology_equiv}
	For every $p \in [1, \infty)$ and $k \in \N$, we have
	\begin{align*}
		H^k_p(M) \cong H^{k}_{\derham}(M) \qquad \text{and} \qquad
		H^k_{p, c}(M) \cong H^{k}_{\derham,c}(M),
	\end{align*}
	where $H^{k}_{\derham}(M)$ is the de Rham cohomology of $M$, and $H^{k}_{\derham,c}(M)$ is the de Rham cohomology of $M$ with compact supports. Moreover, the above isomorphisms are induced by the inclusion maps $C^\infty(\wedge^* M) \hookrightarrow W^{d,p}_\loc(\wedge^* M)$ and $C^\infty_c(\wedge^* M) \hookrightarrow W^{d,p}_c(\wedge^* M)$.
\end{thm}

Theorem \ref{thm:Lp_cohomology_equiv} follows via a standard argument from highly general results of sheaf cohomology. The essential ingredients required for the proof to work are the Poincar\'e lemma from Lemma \ref{lem:poincare_lemma}, the fact that the spaces $W^{d,p}_\loc(\wedge^k U)$ are defined locally, the fact that a $u \in W^{d,p}_\loc(\wedge^0 U)$ with $du = 0$ is locally constant, and the fact that the spaces $W^{d,p}_\loc(\wedge^k U)$ are closed under multiplication by $C^\infty$-functions. For the sake of readers less familiar with sheaf cohomology, we recall here a version of the general result we're relying on, where we try to minimize the use of sheaf-theoretic concepts in the statement.

Let $X$ be a paracompact Hausdorff space. Note that this includes for example all metric spaces. A \emph{presheaf} (of vector spaces) $\cS$ on $X$ is a choice of a vector space $\cS(U)$ for every open set $U \subset X$, combined with restriction maps $u \mapsto u\vert_V \colon \cS(U) \to \cS(V)$ whenever $V \subset U$. The restriction maps are assumed to satisfy the following typical properties of the restriction of functions:
\begin{itemize}
	\item $u\vert_U = u$ if $u \in \cS(U)$;
	\item $(au + bv)\vert_V = a(u\vert_V) + b(v\vert_V)$ if $u, v \in \cS(U)$ and $a, b \in \R$;
	\item $(u\vert_V)\vert_W = u\vert_W$ if $u \in \cS(U)$ and $W \subset V \subset U$.
\end{itemize}
A presheaf $\cS$ is then a \emph{sheaf} if it also satisfies the following two conditions
\begin{enumerate}
	\item[(S1)] \label{enum:sheaf_locality} If $U = \bigcup_i U_i$ and $u \in \cS(U)$ is such that $u\vert_{U_i} = 0$ for every $i$, then $u = 0$.
	\item[(S2)] \label{enum:sheaf_gluing} If $U = \bigcup_i U_i$, and we have elements $u_i \in \cS(U_i)$ such that they coincide on intersections, i.e. $u_i\vert_{U_i \cap U_j} = u_j\vert_{U_i \cap U_j}$ if $U_i$ and $U_j$ intersect, then there exists an element $u \in \cS(U)$ such that $u\vert_{U_i} = u_i$ for every $i$.
\end{enumerate}
Notably, essentially every typical linear function space on a smooth manifold is a presheaf that satisfies (S1): $C^\infty$, $C^0$, $L^p$, $L^p_\loc$, etc. However, (S2) is only satisfied if the definition of the function space is in a sense local: for example $U \mapsto L^p_\loc(U)$ is a sheaf, but $U \mapsto L^p(U)$ fails to be a sheaf since the definition of $L^p(U)$ is global in nature.

A \emph{presheaf morphism} $f \colon \cS \to \cS'$ between two presheaves on $X$ is a collection of linear maps $f \colon \cS(U) \to \cS'(U)$ such that $f(u\vert_V) = (f(u))\vert_V$. A \emph{family of supports} $\Phi$ on $X$ is a collection of closed subsets of $X$ such that
\begin{itemize}
	\item if $C \in \Phi$, then every closed subset of $C$ is also in $\Phi$;
	\item if $C, C' \in \Phi$, then $C \cup C' \in \Phi$;
	\item if $C \in \Phi$, then there's a neighborhood $V$ of $C$ such that $\overline{V} \in \Phi$.
\end{itemize}
If $\cS$ is a sheaf on $X$, then the \emph{support} $\spt u$ of $u \in \cS(X)$ is $X \setminus U$, where $U$ is the largest open set on which $u\vert_U = 0$. We also use $\cS_\Phi(X)$ to denote all elements of $\cS(X)$ with $\spt(u) \in \Phi$.

We can now state the general result we use.

\begin{thm}\label{thm:sheaf_cohomology}
	Let $X$ be a paracompact Hausdorff space, let $\Phi$ be a family of supports, and suppose that we have a sequence
	\begin{equation}\label{eq:reso}
		0 \rightarrow \cS^{-1} \xrightarrow{d_0} \cS^0 \xrightarrow{d_1} \cS^1 \xrightarrow{d_2} \cS^2 \xrightarrow{d_3} \dots
	\end{equation}
	such that the following conditions are satisfied.
	\begin{itemize}
		\item Every $\cS^i$ in \eqref{eq:reso} is a sheaf on $X$. Every $d_i$ is a presheaf morphism.
		\item The sequence \eqref{eq:reso} is exact in the following local sense: if $U \subset X$ is open, $u \in \cS^i(U)$ and $d_{i+1}(u) = 0$, then for every $x \in U$ there is a neighborhood $U_x \subset U$ of $x$ such that $u\vert_{U_x} = d_i(v_x)$ for some $v_x \in \cS^{i-1}(U_x)$.
		\item For $i \geq 0$, the sheaves $\cS^i$ are $\Phi$-soft: that is, for any $C \in \Phi$, any neighborhood $U$ of $C$, and any $u \in \cS^i(U)$, there exists $u' \in \cS^i(X)$ and a neighborhood $V \subset U$ of $C$ such that $u'\vert_V = u\vert_V$.
	\end{itemize}
	Then the cohomology groups $H^i_\Phi(X; \cS^{-1})$ of the sequence of vector spaces
	\[
		0 \rightarrow \cS_\Phi^0(X) \xrightarrow{d_1} \cS_\Phi^1(X) \xrightarrow{d_2} \cS_\Phi^2(X) \xrightarrow{d_2} \dots
	\]
	are determined completely up to isomorphism by the sheaf $\cS^{-1}$ and the family of supports $\Phi$. Moreover, if we have a commutative diagram of sheaves and presheaf morphisms
	\[\begin{tikzcd}
		0 \ar[r] & \cS^{-1} \ar[r, "d_0"] \ar[d, "\id"] & \cS^{0} \ar[r, "d_1"] \ar[d, "s_0"] & \cS^{1} \ar[r, "d_2"] \ar[d, "s_1"] & \dots\\
		0 \ar[r] & \cS^{-1} \ar[r, "d_0'"] & \cT^{0} \ar[r, "d_1'"] & \cT^{1} \ar[r, "d_2'"] & \dots
	\end{tikzcd}\]
	where both rows satisfy the above conditions, then the maps $s_i$ induce an isomorphism between the cohomology groups of $\cS_\Phi^i(X)$ and $\cT_\Phi^i(X)$.
\end{thm}

The proof can be found e.g.\ in \cite[Chapters I-II]{Bredon_book}. In particular, the proof of the above result is completed in \cite[Theorem II.4.1 and Section II.4.2]{Bredon_book} with a different third assumption, where the sheaves are $\Phi$-acyclic instead of $\Phi$-soft. Afterwards, $\Phi$-soft sheaves are defined in \cite[Section II.9]{Bredon_book}, and they are shown to be $\Phi$-acyclic in \cite[Theorem II.9.11]{Bredon_book}. Note that \cite{Bredon_book} defines many terms in terms of so-called stalks and \'etale spaces of sheaves, which we have elected not to discuss here. In most cases any conversion of definitions is straightforward. However, converting the above definition of $\Phi$-soft sheaves to the one in \cite{Bredon_book} is slightly trickier, and requires the paracompactness assumption; the key step is \cite[Theorem II.9.5]{Bredon_book}.

We then prove Theorem \ref{thm:Lp_cohomology_equiv}

\begin{proof}[Proof of Theorem \ref{thm:Lp_cohomology_equiv}]
	We wish to use Theorem \ref{thm:sheaf_cohomology}. For this, we first select $\cS^{-1}$ to be the \emph{constant sheaf} $\cR$ on $M$; that is $\cR(U)$ is the space of locally constant real-valued functions on $U$. We then let $\cT^i$ be the sheaves given by $U \mapsto W^{d,p}_\loc(\wedge^i U)$, and let $\cS^i$ for $i \geq 0$ be the sheaves of smooth differential forms $U \mapsto C^\infty(\wedge^i U)$. The maps $s_i$ are the inclusion maps $C^\infty(\wedge^i U) \hookrightarrow W^{d,p}_\loc(\wedge^i M)$. The maps $d_0$ and $d_0'$ are given by the inclusion $\cR(U) \hookrightarrow C^\infty(\wedge^0 U) \hookrightarrow W^{d,p}_\loc(\wedge^0 U)$, while the other maps $d_i, d_i'$ are given by the (weak) exterior derivative.
	
	All of our chosen $\cS^i$ and $\cT^i$ are indeed sheaves; here, it's crucial to use the local spaces $W^{d,p}_\loc(\wedge^i U)$. Exactness at $\cS^0$ and $\cT^0$ follow from the fact that a real function with zero derivative is locally constant; for a Sobolev version, see e.g.\ \cite[Lemma 1.13]{Heinonen-Kilpelainen-Martio_book}. Exactness at all other points of the sequence follows from the Poincar\'e lemma, or its Sobolev version given in \ref{lem:poincare_lemma}. Finally softness follows from the fact that $C^\infty(\wedge^i U)$ and $W^{d,p}_\loc(\wedge^i U)$ are closed under multiplication by $C^\infty$-functions. Indeed, if $C \subset U$ with $C$ closed and $U$ open, one can multiply any $\omega$ in $C^\infty(\wedge^k U)$ or $W^{d,p}_\loc(\wedge^k U)$ with a suitable smooth cutoff function $\eta \in C^\infty(M)$ satisfying $\spt \eta \subset U$, and hence obtain an extension $\eta \omega$ on $M$ that equals $\omega$ on a neighborhood of $C$.
	
	Hence, Theorem \ref{thm:sheaf_cohomology} applies for any family of supports $\Phi$. We get the compactly supported version by having $\Phi$ be the family of compact subsets of $M$, and the version without supports by having $\Phi$ be the family of all closed subsets of $M$.
\end{proof}

\subsection{Conformal cohomology}
Besides the above cohomology theories, we also need to use a conformal cohomology theory. There are several variations of conformal cohomology theories in use: see e.g. \cite{Donaldson-Sullivan_Acta}, \cite{Goldshtein-Troyanov_DeRham}, and \cite{Kangasniemi-Pankka_PLMS}. Since we generally do not assume higher integrability from our maps, the best suited one for our current application is the one from \cite{Kangasniemi-Pankka_PLMS}. It is the cohomology of the chain complex $\cesobloc(\wedge^* M)$ given by
\begingroup\allowdisplaybreaks\begin{align*}
	\cesobloc(\wedge^0 M) 
	&= \bigcup_{p < \infty} W^{d,p,n}_\loc(\wedge^0 M)\\
	\cesobloc(\wedge^k M),
	&= W^{d, \frac{n}{k}, \frac{n}{k+1}}_\loc(\wedge^k M)
	&&\text{ for } 1 \leq k \leq n-2,\\
	\cesobloc(\wedge^{n-1} M) 
	&= \bigcap_{p > 1} W^{d, \frac{n}{n-1}, p}_\loc(\wedge^{n-1} M),&& \text{ and}\\
	\cesobloc(\wedge^n M)  &= \bigcap_{p > 1} L^p_\loc(\wedge^n M).
\end{align*}\endgroup
The resulting conformal cohomology spaces are denoted $\cehom{k}(M)$.

We also require conformal cohomology with compact supports. In this case the relevant chain complex is
\begin{align*}
	\cesobc(\wedge^0 M) 
	&= \bigcup_{p < \infty} W^{d,p,n}_c(\wedge^0 M),\\
	\cesobc(\wedge^k M) 
	&= W^{d, \frac{n}{k}, \frac{n}{k+1}}_c(\wedge^k M)
	&&\text{ for } 1 \leq k \leq n-2,\\
	\cesobc(\wedge^{n-1} M) 
	&= \bigcap_{p > 1} W^{d, \frac{n}{n-1}, p}_c(\wedge^{n-1} M),&& \text{ and}\\
	\cesobc(\wedge^n M)  &= \bigcap_{p > 1} L^p_c(\wedge^n M).
\end{align*}
The cohomology spaces of this complex are in turn denoted $\cehomc{k}(M)$.

There is also a version of Theorem \ref{thm:Lp_cohomology_equiv} for conformal cohomology. The proof is exactly the same as that of Theorem \ref{thm:Lp_cohomology_equiv}, where the choice of exponents in the cohomology theory is exactly such that Lemma \ref{lem:poincare_lemma} still applies. Hence, we refrain from repeating the argument. Note that for the part of the result involving $\cehom{k}(M)$, a highly detailed explanation of the proof has been given in \cite[Section 4]{Kangasniemi-Pankka_PLMS}; the compactly supported version has however not been stated previously to our knowledge.

\begin{thm}\label{thm:conf_cohomology_equiv}
	For every $k \in \N$, we have
	\begin{align*}
		\cehom{k}(M) \cong H^{k}_{\derham}(M) \qquad \text{and} \qquad
		\cehomc{k}(M) \cong H^{k}_{\derham,c}(M),
	\end{align*}
	where the isomorphisms are induced by the inclusion maps $C^\infty(\wedge^* M) \hookrightarrow \cesobloc(\wedge^* M)$ and $C^\infty_c(\wedge^* M) \hookrightarrow \cesobc(\wedge^* M)$.
\end{thm}

\section{Proof of Theorem \ref{thm:fiber_Hausdorff_bound}}\label{sect:Hencl-koskela}

In this short section, we briefly outline the proof of Theorem \ref{thm:fiber_Hausdorff_bound}. As stated in the introduction, the proof is an obvious generalization of the argument of Hencl and Koskela \cite{Hencl-Koskela_MFD-discr-open}, and is included more for the sake of completeness. We first recall the statement of the result.

\begin{customthm}{\ref{thm:fiber_Hausdorff_bound}}
	Let $\Omega \subset \R^n$ be a domain, $n \ge 2$. Suppose that $f \in W^{1,n}_\loc(\Omega, \R^n)$ is a non-constant mapping of finite distortion and the mapping $f$ has essentially bounded multiplicity. Then for $p \in \bigl[\frac{1}{n-1}, \infty\bigr)$ we have 
	\[
		K_f \in L^p_\loc(\Omega) \implies \cH^\frac{n}{p+1}(f^{-1}\{y\}) = 0 \text{ for all } y \in f(\Omega). 
	\]
\end{customthm}
\begin{proof}
	We may assume $K_f \in L^p(\Omega)$ by considering a countable sequence of subdomains, and we may also assume that $y = 0$. Suppose towards contradiction that $\cH^{n/(p+1)}(f^{-1}\{0\}) > 0$. By our assumption that $p \geq 1/(n-1)$, we obtain that $f$ satisfies the Lusin $(N^{-1})$-condition by \cite[Theorem 4.13]{Hencl-Koskela-book}, and hence $f^{-1} \{0\}$ has zero measure. Using \cite[Theorem 3.2]{Hencl-Koskela_MFD-discr-open} with $u = \abs{f}$ and $\tilde{\Phi}(t) = \log(e + e^t)$ then yields
	\[
		\int_{f^{-1} \B^n(0, \delta)} \frac{\abs{Df}^{\frac{np}{p+1}}}{\abs{f}^{\frac{np}{p+1}} \log(e + \abs{f}^{-1})} = \infty
	\]
	for all small enough $\delta > 0$.
	
	On the other hand, Young's inequality for products yields that
	\begin{equation}\label{eq:young_application}
		\frac{\abs{Df}^{\frac{np}{p+1}}}{\abs{f}^{\frac{np}{p+1}} \log(e + \abs{f}^{-1})}
		\leq \frac{p}{p+1} \frac{\abs{Df}^{n}}{K_f \abs{f}^{n} \log^\frac{p+1}{p}(e + \abs{f}^{-1})}
		+ \frac{1}{p+1} K_f^p.
	\end{equation}
	By our assumption, $K_f$ is $L^p$-integrable over $\Omega$. Moreover, since $f$ has essentially bounded multiplicity, the first term on the right hand side of \eqref{eq:young_application} is also integrable by a change of variables estimate; see \cite[(4.9)]{Hencl-Koskela_MFD-discr-open}. We have hence reached a contradiction, which proves the claim.
\end{proof}

\section{Proof of homological obstructions}

In this section, we prove our main obstruction results: Theorem \ref{thm:homology_obstruction}, Proposition \ref{prop:topo_preimage_restrictions} and Corollary \ref{cor:fiber_homology_obstruction}. We begin by recalling the statement of Proposition \ref{prop:topo_preimage_restrictions} and by giving the short proof.

\begin{customprop}{\ref{prop:topo_preimage_restrictions}}
	Let $f \colon \Omega \to \Omega'$ be a proper, continuous, monotone surjection between open domains in $\R^n$. Then for every $\B^n(y, r) \Subset \Omega'$, we have 
	\begin{align*}
		H_0(f^{-1} \B^n(y, r); \R) &\cong \R,\\
		H_{n-1}(f^{-1} \B^n(y, r); \R) &\cong \{0\},\\
		H_{n}(f^{-1} \B^n(y, r); \R) &\cong \{0\}.
	\end{align*}
\end{customprop}
\begin{proof}
	The case $k = 0$ follows from the fact that if $f \colon X \to Y$ is a continuous monotone surjection between compact spaces, then $f^{-1} C$ is connected for every connected $C \subset X$; see e.g.\ \cite[Corollary 6.1.19]{Engelking_topology}. The case $k = n$ is simply due to the fact that the $n$-homology of any noncompact manifold vanishes.
	
	For the remaining case $k = n-1$, suppose towards contradiction that $H_{n-1}(f^{-1} \B^n(y, r); \R) \ncong \{0\}$. It then follows from Alexander duality that $\tilde{H}^0((\R^n \cup \{\infty\}) \setminus f^{-1} \B^n(y, r); \R) \ncong \{0\}$, where $\tilde{H}^*(X; \R)$ denotes the reduced \u{C}ech cohomology of $X$ with coefficients in $\R$. Now, if $r'$ is such that $r < r' < d(y, \partial \Omega')$,  we have $\tilde{H}^0(f^{-1} (\overline{\B^n(y, r')} \setminus \B^n(y, r)); \R) \ncong \{0\}$ by the reduced Mayer--Vietoris -sequence for the sets $(\R^n \cup \{\infty\}) \setminus f^{-1} \B^n(y, r)$ and $f^{-1} \overline{\B^n(y, r')}$. Since the 0:th \u{C}ech cohomology counts quasicomponents, and since quasicomponents are unions of ordinary connected components, it follows that $f^{-1} (\overline{\B^n(y, r')} \setminus \B^n(y, r))$ is disconnected. This is a contradiction, since $f^{-1} (\overline{\B^n(y, r')} \setminus \B^n(y, r))$ is connected due to the aforementioned result \cite[Corollary 6.1.19]{Engelking_topology}.
\end{proof}

We then prove Theorem \ref{thm:homology_obstruction}. We split it into two sub-theorems: a homological result proven with compactly supported $L^p$-cohomology, and a cohomological result proven with $L^p_\loc$-cohomology. The proofs of these two results are essentially identical. We begin with the homological result.

\begin{lemma}\label{lem:homology_part}
	Let $f \in W^{1,n}(\Omega, \Omega')$ be a proper, continuous, monotone surjection between open domains. Suppose that $k \in \{1, \dots, n-2\}$, and that
	\[
		K_f^{p} \in L^1(\Omega), \quad \text{where }
		p \geq \frac{k}{n-k} \text{ and } p \geq \frac{n-(k+1)}{k+1}.
	\]
	Then
	\[
		H_k(f^{-1} \B^n(y, r); \R) = \{0\} \quad \text{for every } \B^n(y, r) \Subset \Omega'.
	\]
\end{lemma}
\begin{proof}
	Suppose to the contrary that $H_k(f^{-1} \B^n(y, r); \R) \neq \{0\}$ for a given $y$ and $r$. Since $f^{-1} \B^n(y, r)$ is an oriented manifold, we have by Poincar\'e duality that 
	\[
		H^{n-k}_{\derham,c}(f^{-1} \B^n(y, r)) \neq \{0\}
	\]
	We may hence select a $\omega \in C^{\infty}_c(\wedge^{n-k} f^{-1} \B^n(y, r))$ such that the class $[\omega]$ of $\omega$ in $H^{n-k}_{\derham,c}(f^{-1} \B^n(y, r))$ is non-zero. In particular, the $L^1$-case of Theorem \ref{thm:Lp_cohomology_equiv} implies that $\omega$ is not a weak differential of any $\tau' \in W^{d,1}_c(\wedge^{n-k-1} f^{-1} \B^n(y, r))$.
	
	We then consider the push-forward $\push{f} \omega$. By Lemma \ref{lem:pushforward_d_commutation} combined with our assumption that $p \geq k/(n-k)$, we have $\push{f} \omega \in \cesobc(\wedge^{n-k} B^n(y, r))$ and $d\push{f} \omega = \push{f} d\omega = 0$. It follows that $\push{f} \omega$ is in a cohomology class of $\cehomc{n-k}(B^n(y, r))$. By Theorem \ref{thm:conf_cohomology_equiv}, we know that $\cehomc{n-k}(B^n(y, r)) = \{0\}$, and therefore $\push{f} \omega = d\tau$ for some $\tau \in \cesobc(\wedge^{n-k-1} B^n(y, r))$.
	
	Now, $\cesobc(\wedge^{n-k-1} B^n(y, r)) \subset L^{n/(n-k-1)}_c(\wedge^{n-k-1} B^n(y, r))$, where we use our assumption that $k \leq n-2$. Hence, Lemma \ref{lem:pullback_estimate} and the assumption that $f$ is proper yield that
	\[
		f^* \tau \in L^{r}_c(\wedge^{n-k-1} f^{-1} B^n(y,r)), 
		\quad \text{where } r = \frac{n}{(n - k - 1)(1 + p^{-1})}.
	\]
	Our assumption that $p \geq (n-k-1)/(k+1)$ can be re-arranged as $1 + p^{-1} \leq n/(n - k - 1)$. Hence, $r \geq 1$, and it also follows from Lemma \ref{lem:pullback_commutation_lemma} that $d f^* \tau = f^* d \tau = f^* \push{f} \omega = \omega$. This contradicts the fact that $\omega$ is not a weak differential of any $\tau' \in W^{d,1}_c(\wedge^{n-k-1} f^{-1} \B^n(y, r))$. We hence conclude that $H_k(f^{-1} \B^n(y, r); \R) = \{0\}$, completing the proof.
\end{proof}

We then give the cohomological version of Lemma \ref{lem:homology_part}. Note that this version has different assumptions on the integrability of $K_f$.

\begin{lemma}\label{lem:cohomology_part}
	Let $f \in W^{1,n}(\Omega, \Omega')$ be a proper, continuous, monotone surjection between open domains. Suppose that $k \in \{2, \dots, n-1\}$, and that
	\[
		K_f^{p} \in L^1(\Omega), \quad \text{where }
		p \geq \frac{n-k}{k} \text{ and } p \geq \frac{k-1}{n-(k-1)}.
	\]
	Then
	\[
		H^k(f^{-1} \B^n(y, r); \R) = \{0\} \quad \text{for every } \B^n(y, r) \Subset \Omega'.
	\]
\end{lemma}
\begin{proof}
	The proof is essentially the same as that of Lemma \ref{lem:homology_part}. Indeed, instead of starting with a form $\omega \in C^\infty_c(\wedge^{n-k} f^{-1} B^n(y,r))$, we use the de Rham theorem to conclude that $H^k_{\derham}(f^{-1} \B^n(y, r)) \neq \{0\}$, and start with a form $\omega \in C^\infty(\wedge^{k} f^{-1} B^n(y,r))$ with $[\omega] \neq [0]$. The change from an $(n-k)$-form to a $k$-form causes the changes in our assumptions on $k$ and $p$. The result then follows by repeating the rest of the argument of Lemma \ref{lem:homology_part}, where all compactly supported cohomology theories are replaced with the corresponding theory without compact supports, and integrability results are applied locally using the continuity of $f$.
\end{proof}

Now, Theorem \ref{thm:homology_obstruction} follows from Lemmas \ref{lem:homology_part} and \ref{lem:cohomology_part}. We recall the statement and give the few remaining details.

\begin{customthm}{\ref{thm:homology_obstruction}}
	Let $f \in W^{1,n}(\Omega, \Omega')$ be a proper, continuous, monotone surjection between open domains in $\R^n$. Suppose that $k \in \{1, \dots, n-2\}$, and that
	\[
		K_f \in L^p_\loc(\Omega), \quad \text{where }
		p = \begin{cases}
			\frac{n-(k+1)}{k+1},& 1 \leq k < \frac{n}{2},\\
			1,& k = \frac{n}{2},\\
			\frac{k-1}{n-(k-1)},& \frac{n}{2} < k \leq n-2.
		\end{cases}
	\]
	Then
	\[
		H_k(f^{-1} \B^n(y, r); \R) = \{0\} \quad \text{for every } \B^n(y, r) \Subset \Omega'.
	\]
\end{customthm}
\begin{proof}
	If $1 \leq k < n/2$, then $p = (n-k-1)/(k+1) > k/(n-k)$, and hence $H_k(f^{-1} \B^n(y, r); \R) = \{0\}$ by Lemma \ref{lem:homology_part}, which is the desired result. If on the other hand $n/2 < k \leq n-2$ (or if we are in the unnecessary case $k = n-1$), then we similarly have $p = \left(k-1\right)/\left(n-k+1\right) > (n-k)/k$, in which case $H^k(f^{-1} \B^n(y, r); \R) = \{0\}$ by Lemma \ref{lem:cohomology_part}. Since $H^k(f^{-1} \B^n(y, r); \R) \cong H_k(f^{-1} \B^n(y, r); \R)$ by the universal coefficient theorem, we hence have our claim also in this case. The final case is $k = n/2$: in this case, our definition also gives $p = 1 = k/(n-k) > (n-k-1)/(k+1)$, and therefore Lemma \ref{lem:homology_part} yields the claim.
\end{proof}

To end this section, we recall the version of the result for fibers given in Corollary \ref{cor:fiber_homology_obstruction}, and give the short proof.

\begin{customcor}{\ref{cor:fiber_homology_obstruction}}
	Let $f \in W^{1,n}(\Omega, \Omega')$ be a proper, continuous, monotone surjection between open domains in $\R^n$. Let $k \in \{1, \dots, n\}$. Moreover, if $k \leq n-2$, suppose also that
	\[
	K_f \in L^p_\loc(\Omega), \quad \text{where }
	p = \begin{cases}
		\frac{n-(k+1)}{k+1},& 1 \leq k < \frac{n}{2},\\
		1,& k = \frac{n}{2},\\
		\frac{k-1}{n-(k-1)},& \frac{n}{2} < k \leq n-2.
	\end{cases}
	\]
	If $y \in \Omega'$ is such that $f^{-1}\{y\}$ is a neighborhood retract, then $H_k(f^{-1} \{y\}; \R) = \{0\}$.
\end{customcor}
\begin{proof}
	Let $U \subset \R^n$ be a neighborhood of $f^{-1} \{y\}$ and let $r \colon U \to f^{-1}\{y\}$ be a retraction. The sets $U_i = f^{-1} B^n(y, i^{-1})$ for large enough $i$ form a sequence of pre-compact neighborhoods of $f^{-1} \{y\}$ with $\overline{U_{i+1}} \subset U_i$ and $\bigcap_i U_i = f^{-1} \{y\}$. It follows that $U_i \subset U$ for some $i$. Now, if $\iota^i \colon f^{-1}\{y\} \hookrightarrow U_i$ and $\kappa^i \colon U_i \hookrightarrow U$ are inclusions and $c \in H_k(f^{-1} \{y\}; \R)$, then Theorem \ref{thm:homology_obstruction} yields $c = \push{r} \push{\kappa^i} \push{\iota^i} c = \push{r} \push{\kappa^i} 0 = 0$, which yields the claim.
\end{proof}

\section{The example with circular fibers}

\enlargethispage{\baselineskip}We begin by recalling the statement of Theorem \ref{thm:Bing_meets_Sobolev_again}.

\begin{customthm}{\ref{thm:Bing_meets_Sobolev_again}}
	There exists a map $h \colon \R^3 \to \R^3$ with the following properties.
	\begin{itemize}
		\item The map $h$ is topologically monotone, proper, and surjective onto $\R^3$.
		\item The map $h$ is locally Lipschitz, and $J_h$ is positive almost everywhere. Hence, $h$ is a mapping of finite distortion.
		\item We have $K_h \in L^p_\loc(\R^3, \R^3)$ for every $p < 1/2$, but $K_h \notin L^{1/2}_\loc(\R^3, \R^3)$.
		\item The fibers $h^{-1}\{0\}$ and $h^{-1}\{-e_x\}$ are bilipschitz equivalent with $\S^1$. The fibers $h^{-1}\{-te_x\}$ for $t \in (0, 1)$ are bilipschitz equivalent with $\S^1 \vee \S^1$. The fibers $h^{-1}\{-te_x\}$ for $t \in (1, \infty)$ are bilipschitz equivalent with $[0, 1]$. For all other values $y \in \R^3 \setminus \{-te_x, t \geq 0\}$, the fiber $h^{-1}\{y\}$ is a point.
	\end{itemize}
\end{customthm}

We then begin the construction of the map $h$ as above. We use cylindrical coordinates $(r, \theta, z)$ on the domain side, where $r \geq 0$ and $\theta \in (-\pi, \pi]$. On the target side, we use standard Euclidean coordinates $(x, y, z)$. We also use $\sgn{t} = t/\abs{t}$ to denote the sign of a real number $t \in \R$, with $\sgn(0) = 0$.

We partition the domain into a family of square torii $T_c$, $c \in [0, \infty)$, defined by
\[
	T_c = \{(r, \theta, z) \in \R^3: \abs{r-1} + \abs{z} = c\}.
\]
When $c = 0$, $T_c$ is the circle defined by $r = 1$ and $z = 0$. For $c \in (0, 1)$, $T_c$ is a sharp-cornered topological torus. When $c = 1$, the hole in the center of the torus gets closed, and as $c$ increases above $1$, the surface becomes topologically $\S^2$.

When $c \in [0, 1]$, we map the slices $T_{c, \theta} = T_c \cap \{(r, \theta, z) \in \R^3 : r \in [0, \infty), z \in \R\}$ by the composition of the following three maps.
\begin{itemize}
	\item Place the square $T_{c, \theta}$ into the $xy$-plane, centered at the origin, with the map 
	\[
		(r, \theta, z) \mapsto (r-1)e_x + ze_y. 
	\]
	\item Scale down uniformly by a factor of $\abs{\theta}/\pi$, with the center of scaling at the tip $(-c, 0)$. This map is given by
	\[
		(x,y) \mapsto \frac{\abs{\theta}}{\pi}(x - (-c), y) + (-c, 0).
	\]
	\item Then fix the tip of the square at $(-c, 0)$, and move the other tip at $(c, 0)$ into the $z$-direction so that the slope of the square becomes $(\abs{\theta}/\pi - 1) \sgn(\theta)$. That is, the relevant map is
	\[
		(x,y) \mapsto \left(x, y, \frac{(\abs{\theta} - \pi) \sgn(\theta)}{\pi} (x + c)\right)
	\] 
	Note that while the map of this step has a discontinuity at $\theta = 0$, the previous step will cancel out this discontinuity. 
\end{itemize}
\enlargethispage{\baselineskip}This defines our map $h$ in the region of $\R^3$ where $\abs{r-1} + \abs{z} \leq 1$. See the following Figures \ref{fig:h_center}-\ref{fig:h_closed_torus} for an illustration of the resulting map $h$. 

\begin{figure}[hb]
	\includegraphics[width=\textwidth]{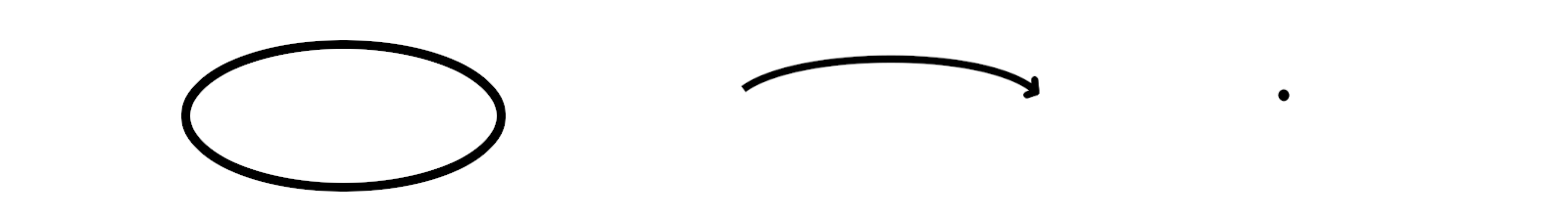}
	
	\caption{\small The set $T_0$ is just the unit circle in the $xy$-plane. The map $h$ collapses it to the origin.}\label{fig:h_center}
	
	\includegraphics[width=\textwidth]{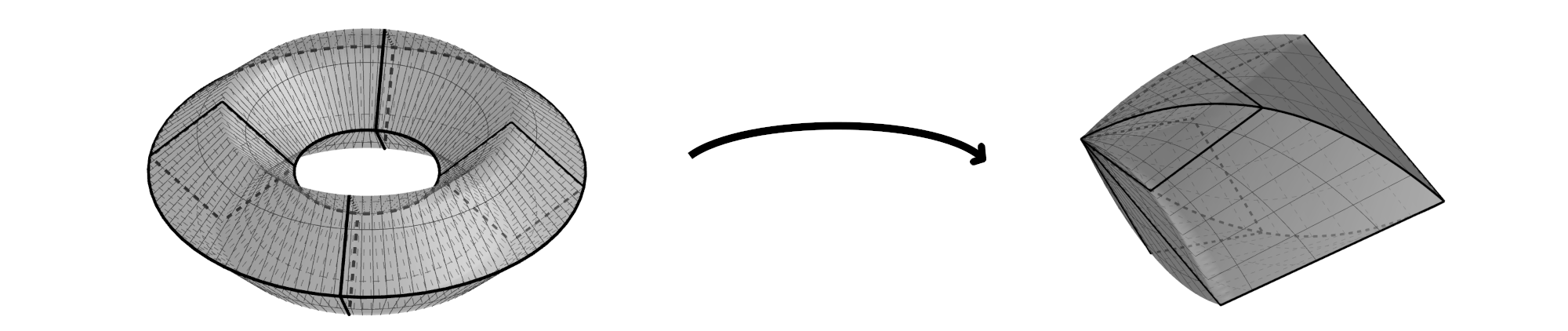}
	
	\caption{\small When $0 < c < 1$, the set $T_c$ is a square torus around the circle $T_0$. It is mapped into a surface centered at the origin, with the size of the image increasing with $c$. The inner ring of the torus and one of the square cross-sections get mapped to the single point at the tip of the surface.}
	\label{fig:h_torus}
	
	\medskip
	
	\includegraphics[width=\textwidth]{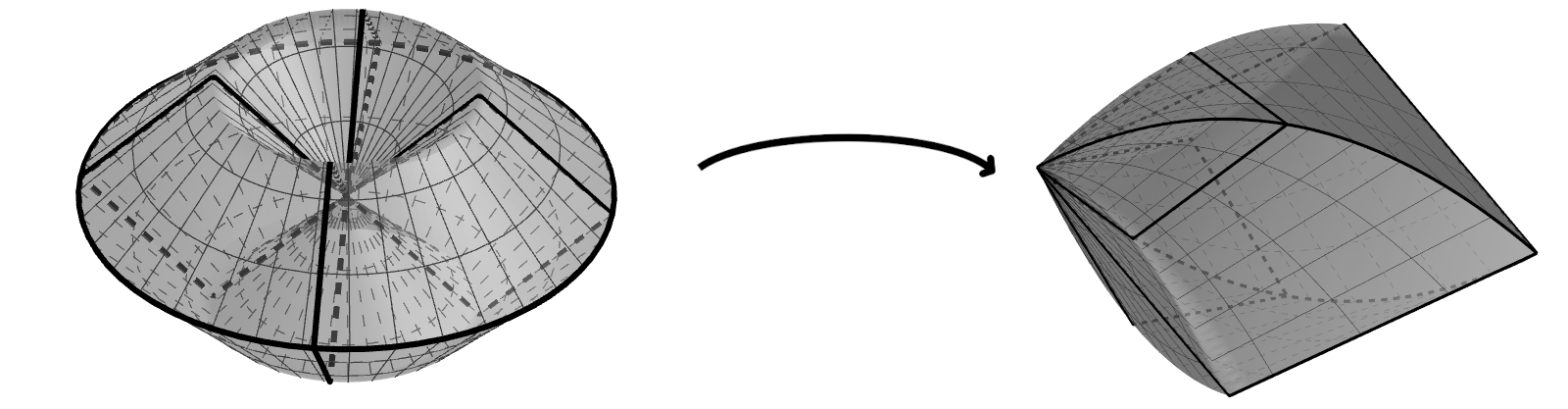}
	
	\caption{\small When $c = 1$, the torus $T_1$ gets closed in the middle. However, our previous process of defining $h$ remains valid, since the inner ring of the torus was mapped to a single point.}
	\label{fig:h_closed_torus}
\end{figure}

Moreover, by computing the composition of the three component maps of $h$ and applying $c = \abs{r-1} + \abs{z}$, we get an explicit formula for the map $h$. That is,
\begin{multline}\label{eq:h_def_1}
	h(r, \theta, z) = \left(\frac{\abs{\theta}}{\pi} (r-1 + \abs{r-1} + \abs{z}) - (\abs{r-1} + \abs{z})\right)e_x\\
	+ \frac{\abs{\theta}}{\pi} z e_y + \frac{(\pi - \abs{\theta})\theta}{\pi^2}(r-1 + \abs{r-1} + \abs{z})e_z,
\end{multline}
which applies when $\abs{r-1} + \abs{z} \leq 1$.

When $c > 1$, the slice $T_{c, \theta}$ is no longer a complete square, but instead gets cut off at the $z$-axis. Hence, we modify $T_{c, \theta}$ into a square $T_{c, \theta}'$. We do this by uniformly scaling the two cut-off sides of $T_{c, \theta}$; see Figure \ref{fig:square_conversion} for an illustration. Afterwards, we apply the same map as in the cases $c < 1$ to $T_{c, \theta}'$, where in the first step $(r, \theta, z) \mapsto (r-1)e_x + ze_y$ we use negative values of $r$ for the part of $T_{c, \theta}'$ that passes the $z$-axis, in order to preserve the shape of $T_{c, \theta}'$. The resulting map $h$ is shown in Figure \ref{fig:h_fat_torus}.

\begin{figure}[h]
	\begin{tikzpicture}[scale=0.8]
		\draw[gray] (0, 0.5) -- (1, 1.5);
		\draw (1, 1.5) -- (1.75, 0.75) node[anchor=west] {$T_{c, \theta}$} -- (2.5, 0) -- (1, -1.5);
		\draw[gray] (1, -1.5) -- (0, -0.5);
		\draw [-stealth] (0,-1.6) -- (0,1.6) node[anchor=west] 
			{$z$};
		
		\draw [-stealth] (3,0) -- (4,0);
		
		\draw[gray] (6, -1.5) -- (4.5, 0) -- (6, 1.5);
		\draw (6, 1.5) -- 
		(6.75, 0.75) node[anchor=west] {$T_{c, \theta}'$}
		-- (7.5, 0) -- (6, -1.5);
		\draw [-stealth] (5,-1.6) -- (5,1.6) node[anchor=west] 
			{$z$};
	\end{tikzpicture}

	\caption{\small How the cross-section sets $T_{c, \theta}$ are converted into squares $T_{c, \theta}'$ when $c > 1$. The gray part is scaled linearly, while the black part remains unchanged.}
	\label{fig:square_conversion}
	
	\medskip
	
	\includegraphics[width=\textwidth]{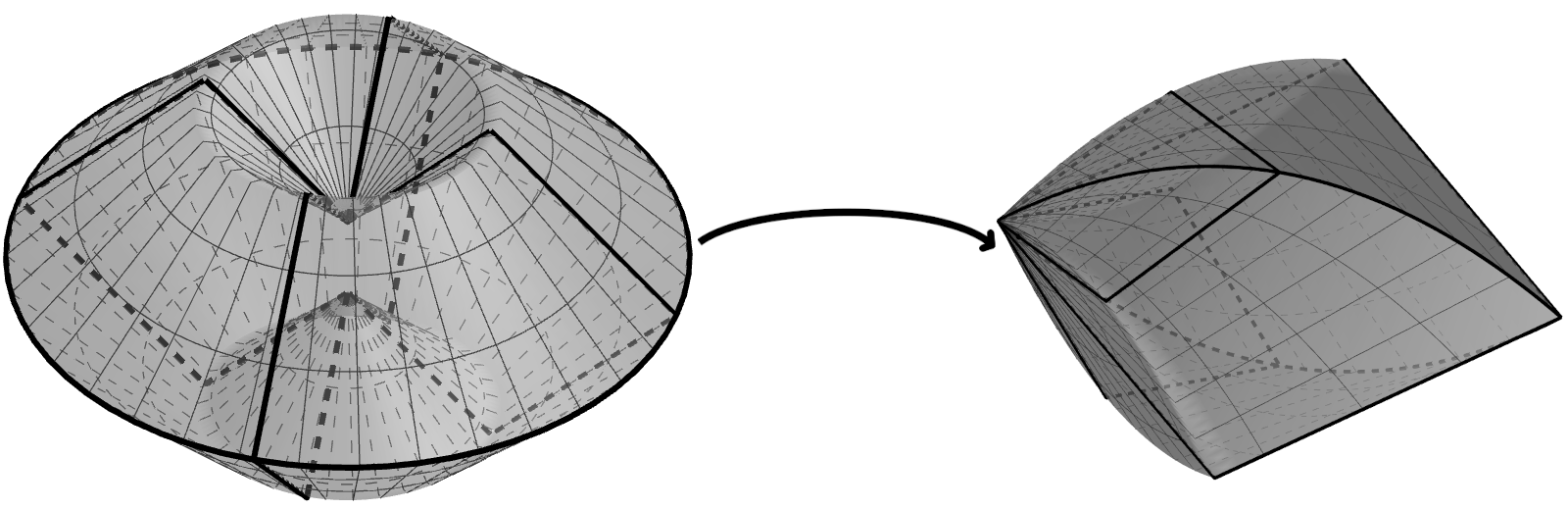}
	
	\caption{\small The resulting map on $T_c$ when $c > 1$. Now only one of the cut-off squares gets mapped to the tip on the image side.}
	\label{fig:h_fat_torus}
\end{figure}

We can again obtain an explicit formula for the resulting map $h$. When $r \leq 1$ and $r \leq \abs{z}$, the formula is
\begin{multline}\label{eq:h_def_2}
	h(r, \theta, z) = \left(\frac{\abs{\theta}}{\pi} r - 1 \right)(\abs{z} - r + 1)e_x\\
	+ \frac{\abs{\theta}}{\pi} r(\abs{z} - r + 1)\sgn(z) e_y 
	+ \frac{(\pi - \abs{\theta})\theta}{\pi^2}r(\abs{z} - r + 1)e_z .
\end{multline}
Everywhere else, including when $r \geq 1$, the map $h$ is given by  the same formula \eqref{eq:h_def_1}. 

Hence, our map $h$ is now defined on all of $\R^3$. We remark that our choice of using square torii, as well as our choice for the shape of the sets $h(T_c)$, are all motivated by the fact that the formulas \eqref{eq:h_def_1} and \eqref{eq:h_def_2} we get for $h$ are relatively simple polynomials in $r$, $\theta$ and $z$. This vastly simplifies the computations required for the proof of Theorem \ref{thm:Bing_meets_Sobolev_again}.

We note that the set of points $B_h$ where $h^{-1}\{h(x)\} \neq \{x\}$ consists of the disk $\{(r, \theta, z) : z = 0, r \leq 1\}$ combined with the half-plane $\{(r, \theta, z) : \theta = 0, r \geq 0\}$. An illustration of the non-trivial fibers of $h$ is given in Figure \ref{fig:fibers}.

\begin{figure}[h]
	\includegraphics[width={0.6\textwidth}]{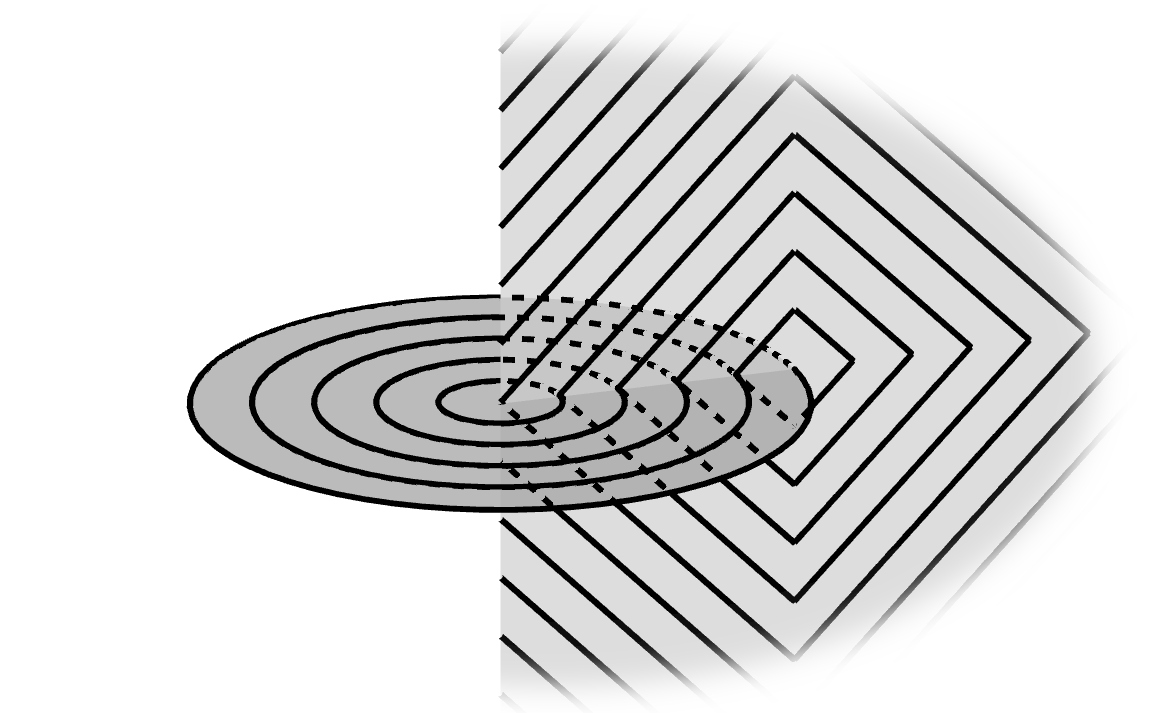}
	
	\caption{\small The set where the map $h$ is not a homeomorphism, with some of the fibers illustrated. The fibers $h^{-1}\{-ce_x\}$ with $0 < c < 1$ are figure-eights that interpolate between two linked loops. For $c > 1$, the fibers stop at the $z$-axis, and are hence topologically equivalent to a line segment.}
	\label{fig:fibers}
\end{figure}

We then verify that our map $h$ satisfies the required conditions.

\begin{proof}[Proof of Theorem \ref{thm:Bing_meets_Sobolev_again}]
	As stated above, our map $h$ is given by \eqref{eq:h_def_2} when $r \leq \min(1, \abs{z})$ and by \eqref{eq:h_def_1} elsewhere, where we assume $r \geq 0$ and $\theta \in (-\pi, \pi]$. It is clear from the geometry of the construction of $h$ that $h$ is a continuous surjection, that the fibers of $h$ are as specified, and that $h$ is hence topologically monotone.
	
	It is clear from the formulas \eqref{eq:h_def_1} and \eqref{eq:h_def_2} that $h_x$, $h_y$, and $h_z$ are absolutely continuous on every line of the type $\{(r_0, \theta_0, z) : z \in \R\}$, $\{(r_0, \theta, z_0) : \theta \in [-\pi, \pi] \}$, and $\{(r, \theta_0, z_0) : r \in [0, \infty)\}$. Hence, the partial derivatives $\partial_r (h_x, h_y, h_z)$, $\partial_\theta (h_x, h_y, h_z)$, and $\partial_z (h_x, h_y, h_z)$ exist for almost all $r$, $\theta$ and $z$. We may also easily compute the partial derivatives from \eqref{eq:h_def_1} and \eqref{eq:h_def_2}; when $r \geq \min(1, \abs{z})$, they are given by
	\begin{multline}\label{eq:h_derivative_1}
		D_{r, \theta, z}^{x,y,z} h(r, \theta, z) = \begin{bmatrix}
			\partial_r h_x(r, \theta, z) & \partial_\theta h_x(r, \theta, z) & \partial_z h_x(r, \theta, z)\\
			\partial_r h_y(r, \theta, z) & \partial_\theta h_y(r, \theta, z) & \partial_z h_y(r, \theta, z)\\
			\partial_r h_z(r, \theta, z) & \partial_\theta h_z(r, \theta, z) & \partial_z h_z(r, \theta, z)
		\end{bmatrix} \\
		= \begin{bmatrix}
			\frac{\abs{\theta}}{\pi} + \left(\frac{\abs{\theta}}{\pi} - 1\right)\sgn(r-1)&
			\frac{\sgn(\theta)}{\pi}(r-1 + \abs{r-1} + \abs{z})&
			\left( \frac{\abs{\theta}}{\pi} - 1 \right) \sgn(z)\\
			0&
			\frac{\sgn(\theta)}{\pi}z&
			\frac{\abs{\theta}}{\pi}\\
			\frac{(\pi - \abs{\theta})\theta}{\pi^2} \left( \sgn(r-1) + 1\right)&
			\frac{\pi - 2\abs{\theta}}{\pi^2}(r-1 + \abs{r-1} + \abs{z})&
			\frac{(\pi - \abs{\theta})\theta}{\pi^2} \sgn(z)
		\end{bmatrix},
	\end{multline}
	and when $r \leq \min(1, \abs{z})$, they are given by 
	\begin{multline}\label{eq:h_derivative_2}
		D_{r, \theta, z}^{x,y,z} h(r, \theta, z) \\
		= \begin{bmatrix}
			\frac{\abs{\theta}}{\pi} \left(\abs{z} - 2r+1\right) +1 &
			\frac{\sgn(\theta)}{\pi}r(\abs{z} - r + 1)&
			\left( \frac{\abs{\theta}}{\pi}r - 1 \right) \sgn(z)\\
			\frac{\abs{\theta}}{\pi} \left(\abs{z} - 2r+1\right) \sgn(z)&
			\frac{\sgn(\theta)}{\pi}r(\abs{z} - r + 1)\sgn(z)&
			\frac{\abs{\theta}}{\pi}r\\
			\frac{(\pi - \abs{\theta})\theta}{\pi^2} \left(\abs{z} - 2r+1\right) &
			\frac{\pi - 2\abs{\theta}}{\pi^2}r(\abs{z} - r + 1)&
			\frac{(\pi - \abs{\theta})\theta}{\pi^2}r \sgn(z)
		\end{bmatrix}.
	\end{multline}
	
	We then observe that $\partial_r h$, $\partial_z h$, and $r^{-1} \partial_\theta h$ are locally essentially bounded. Indeed, the only one for which this is not entirely obvious from \eqref{eq:h_derivative_1} and \eqref{eq:h_derivative_2} is $r^{-1} \partial_\theta h$. However, in the case $r \leq \min(1, \abs{z})$ we have a common factor $r$ in $\partial_\theta h$, in the case $\abs{z} \leq r \leq 1$ we have $\abs{\partial_{\theta} h} \leq \pi^{-1} \left( 2(r - 1 + \abs{r-1}) + 3\abs{z}\right) = 3\pi^{-1} \abs{z} \leq 3\pi^{-1} r$, and in the case $r \geq 1$ the coefficient $r^{-1}$ in $r^{-1} \partial_\theta h$ is bounded from above by 1. Now, since $\partial_r h$, $\partial_z h$, and $r^{-1} \partial_\theta h$ are locally $L^\infty$, and since we have absolute continuity on every line of the type $\left\{(r, \theta_0, z_0)\right\}$, $\left\{(r_0, \theta, z_0)\right\}$ and $\left\{(r_0, \theta_0, z)\right\}$, it follows from a standard path integral estimate argument that $h$ is locally Lipschitz.
	
	It now remains to compute the Jacobian $J_h$ of $h$. Note that we need an extra $r^{-1}$-term in front of the determinant of $D_{r, \theta, z}^{x,y,z} h$ to get the standard Jacobian, since $dr \wedge d\theta \wedge dz = r^{-1} dx \wedge dy \wedge dz$. We split to the three cases $\abs{z} \leq r \leq 1$, $r \leq \min(1, \abs{z})$, and $r \geq 1$. 
	
	In the case $\abs{z} \leq r \leq 1$, we easily compute using \eqref{eq:h_derivative_1} that
	\begin{align*}
		J_h(r, \theta, z) &= \frac{1}{r} \det(D_{r, \theta, z}^{x,y,z} h(r, \theta, z))\\
		&= \frac{1}{r} \det
		\begin{bmatrix}
			1&
			\frac{\sgn(\theta)}{\pi}\abs{z}&
			\left( \frac{\abs{\theta}}{\pi} - 1 \right) \sgn(z)\\
			0&
			\frac{\sgn(\theta)}{\pi}z&
			\frac{\abs{\theta}}{\pi}\\
			0 &
			\frac{\pi - 2\abs{\theta}}{\pi^2}\abs{z}&
			\frac{(\pi - \abs{\theta})\theta}{\pi^2} \sgn(z)\\
		\end{bmatrix}\\
		&= \frac{\abs{z}}{r} \frac{\abs{\theta}^2}{\pi^3}.
	\end{align*}
	In the case $r \leq \min(1, \abs{z})$, we similarly get $J_h$ by dividing the determinant of \eqref{eq:h_derivative_2} by $r$. Even though the matrix appears complicated, large parts of the first and third column are multiples of each other, leading to a great degree of simplification with the relatively tidy result
	\begin{align*}
		J_h(r, \theta, z)
		&= \frac{1}{r} \det (D_{r, \theta, z}^{x,y,z} h(r, \theta, z))\\
		&= \frac{(1+\abs{z} - r)^2 \abs{\theta}^2}{\pi^3}.
	\end{align*}
	The remaining case $r \geq 1$, computed using \eqref{eq:h_derivative_1}, yields the most complicated $J_h$. Namely, the result in this case is
	\begin{align*}
		J_h(r, \theta, z) &= \frac{1}{r} \det
		\begin{bmatrix}
			2\frac{\abs{\theta}}{\pi} - 1&
			\frac{\sgn(\theta)}{\pi}(2r - 2 + \abs{z})&
			\left( \frac{\abs{\theta}}{\pi} - 1 \right) \sgn(z)\\
			0&
			\frac{\sgn(\theta)}{\pi}z&
			\frac{\abs{\theta}}{\pi}\\
			2\frac{(\pi - \abs{\theta})\theta}{\pi^2}&
			\frac{\pi - 2\abs{\theta}}{\pi^2}(2r - 2 + \abs{z})&
			\frac{(\pi - \abs{\theta})\theta}{\pi^2} \sgn(z)
		\end{bmatrix}\\
		&= \frac{\abs{\theta}}{\pi^4 r}\left( (4\smallabs{\theta}^2 - 4\pi \smallabs{\theta} + 2\pi^2)(r-1) + (2\smallabs{\theta}^2 - 3\pi\smallabs{\theta} + 2\pi^2) \abs{z} \right).
	\end{align*}
	From the computed values of $J_h$, we see that $J_h > 0$ a.e.\ in $\R^3$; the fact that $J_h$ does not change sign was also to be expected by the monotonicity of $h$. Hence, we conclude that $h$ is a mapping of finite distortion. 
	
	Since $h$ is locally Lipschitz, we obtain that $K_h \leq C J_h^{-1}$ a.e.\ locally. In the region $\abs{z} \leq r \leq 1$ we have $J_h^{-1} = \pi^3 r \abs{z}^{-1} \abs{\theta}^{-2}$, which is locally $L^p$-integrable for $p < 1/2$. When $r \leq \min(1, \abs{z})$, we estimate by the arithmetic-geometric mean inequality that 
	\[
		J_h^{-1} = \pi^3 \abs{\theta}^{-2} ((1-r) + \abs{z})^{-2} \leq \frac{\pi^3}{4} \abs{\theta}^{-2} (1-r)^{-1} \abs{z}^{-1},
	\]
	where the upper bound is also clearly locally $L^p$-integrable for $p < 1/2$. Moreover, in the case $r \geq 1$, we can similarly estimate
	\[
		J_h^{-1} \leq \pi^4 r \abs{\theta}^{-1} \left(\pi^2 (r-1) + \frac{7\pi^2}{8} \abs{z}\right)^{-1} \leq \sqrt{\frac{2}{7}} \pi^2 r \abs{\theta}^{-1} (r-1)^{-\frac{1}{2}} \abs{z}^{-\frac{1}{2}},
	\]
	where the upper bound is in fact locally $L^p$-integrable for all $p < 1$. We conclude that $K_h \in L^p_\loc(\R^3)$ for $p < 1/2$. Moreover, in the region $\abs{z} \leq r \leq 1$ we have $\norm{Dh} \geq \abs{\partial_r h} = 1$, and $J_h^{-1}$ is not locally $L^{1/2}$-integrable in this region near the plane $\{\theta = 0\}$. Hence, $K_h \notin L^{1/2}_\loc(\R^3)$.
\end{proof}


\bibliographystyle{abbrv}
\bibliography{sources}

\end{document}